\def\draft{n}
\documentclass[12pt]{amsart}
\usepackage[headings]{fullpage}
\usepackage{amssymb,epic,eepic,epsfig,amsbsy,amsmath,amscd,color}
\usepackage[all]{xy}
\usepackage{graphicx}
\usepackage{texdraw}
\usepackage{url}
\usepackage{bbm}
\usepackage{mathrsfs}

\def\printname#1{
        \if\draft y
                \smash{\makebox[0pt]{\hspace{-0.5in}
                        \raisebox{8pt}{\tt\tiny #1}}}
        \fi
}

\def\printname#1{
        \if\draft y
                \smash{\makebox[0pt]{\hspace{-0.5in}
                        \raisebox{8pt}{\tt\tiny #1}}}
        \fi
}

\newlength{\standardunitlength}
\catcode`\@=11
\long\def\@makecaption#1#2{%
     \vskip 10pt

\setbox\@tempboxa\hbox{
       \small\sf{\bfcaptionfont #1. }\ignorespaces #2}%
     \ifdim \wd\@tempboxa >\captionwidth {%
         \rightskip=\@captionmargin\leftskip=\@captionmargin
         \unhbox\@tempboxa\par}%
       \else
         \hbox to\hsize{\hfil\box\@tempboxa\hfil}%
     \fi}
\font\bfcaptionfont=cmssbx10 scaled \magstephalf
\newdimen\@captionmargin\@captionmargin=2\parindent
\newdimen\captionwidth\captionwidth=\hsize
\catcode`\@=12

\def\lbl#1{\label{#1}\printname{#1}}

                        \theoremstyle{plain}

\newtheorem{theorem}{Theorem}[section]
\newtheorem{thm}{Theorem}
\newtheorem{lemma}[theorem]{Lemma}
\newtheorem{corollary}[theorem]{Corollary}
\newtheorem{proposition}[theorem]{Proposition}

\theoremstyle{definition}

\newtheorem{remark}[theorem]{Remark}

\def\BC{\mathbb C}
\def\BN{\mathbb N}
\def\BZ{\mathbb Z}

\def\la{\langle}
\def\ra{\rangle}

\def\tS{\tilde {\fS}}

\def\cS{\mathscr S}
\def\ot{\otimes}

\def\cE{\mathcal E}
\def\cF{\mathcal F}

\def\bk{\mathbf k}

\def\oD{{\mathring\Delta}}

\def\cP{\mathcal P}
\def\tD{\tilde D}

\def\fS{\mathfrak S}
\def\bfS{\overline{\fS}}

\def\D{\Delta}

\def\cY{\mathcal Y}
\def\ev{{\mathrm{ev}}}

\def\embed{\hookrightarrow}

\def\bbS{\overline \Sigma}

\def\cY{\mathcal Y}

\def\ev{{\mathrm{ev}}}

\def\embed{\hookrightarrow}

\def\pS{\partial \Sigma}
\def\btau{\bar \tau}
\def\ooS{\mathring {\cS}}
\def\Teich{Teichm\"uller }

\def\Tr{\mathrm{Tr}}

\DeclareMathOperator{\tr}{\mathrm tr}

\def\al{\alpha}
\def\ve{\varepsilon}
\def\be { \begin{equation} }
\def\ee { \end{equation} }

\newcommand\no[1]{}

      \def\nc{\newcommand}

                  \nc\FI[2]{\begin{figure}
    \begin{center}\input{#1.pstex_t}\end{center}
    \caption{#2}
    \lbl{#1}
  \end{figure}}
\nc\FIG[3]{\begin{figure}
    \includegraphics[#3]{#1.eps}
    \caption{#2}
    \lbl{fig:#1}
    \end{figure}}
\nc\FF[3]{\begin{figure}
    \includegraphics[#3]{#1.eps}
    \caption{#2}
    \lbl{#1}
    \end{figure}}
    \nc\FIGc[3]{\begin{figure}[htpb]
    \includegraphics[height=#3]{#1.eps}
    \caption{#2}
    \lbl{fig:#1}
    \end{figure}}

    \nc\FIGh[3]{\begin{figure}[htpb]
    \includegraphics[height=#3]{draws/#1.eps}
    \caption{#2}
    \lbl{fig:#1}
    \end{figure}}

\newcommand\incl[2]{{\includegraphics[height=#1]{#2.eps}}}

\def\leftve{\raisebox{-8pt}{\incl{.8 cm}{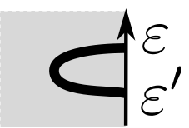}}}
\def\leftup{\raisebox{-8pt}{\incl{.8 cm}{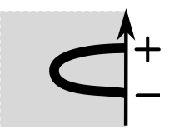}}}
\def\loopup{\raisebox{-8pt}{\incl{.8 cm}{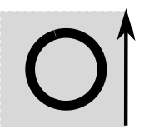}}}
\def\leftupp{\raisebox{-8pt}{\incl{.8 cm}{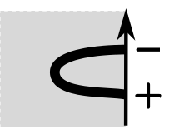}}}
\def\leftupPP{\raisebox{-8pt}{\incl{.8 cm}{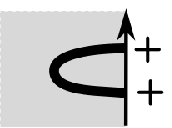}}}
\def\leftupNN{\raisebox{-8pt}{\incl{.8 cm}{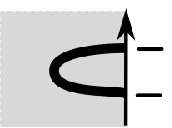}}}

\def\leftved{\raisebox{-8pt}{\incl{.8 cm}{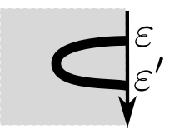}}}
\def\leftvetwist{\raisebox{-8pt}{\incl{.8 cm}{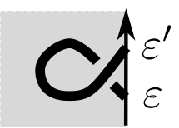}}}
\def\leftvedl{\raisebox{-8pt}{\incl{.8 cm}{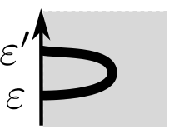}}}
\def\emptyr{\raisebox{-8pt}{\incl{.8 cm}{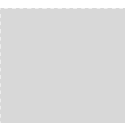}}}
\def\emptys{\raisebox{-8pt}{\incl{.8 cm}{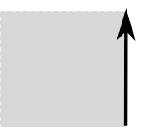}}}

\def\RthreeD{\left(  \raisebox{-10pt}{\incl{1 cm}{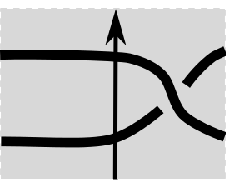}} \right) }
\def\RthreeDpositive{\left(  \raisebox{-10pt}{\incl{1 cm}{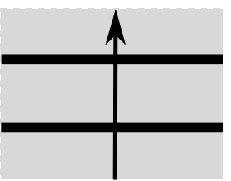}} \right) }
\def\RthreeDnegative{\left(  \raisebox{-10pt}{\incl{1 cm}{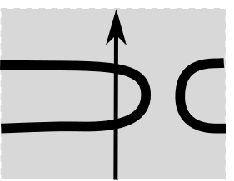}} \right) }
\def\RthreeDp{\left( \raisebox{-10pt}{\incl{1 cm}{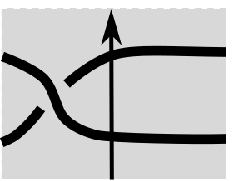}}\right)}
\def\RthreeDpnegative{\left(  \raisebox{-10pt}{\incl{1 cm}{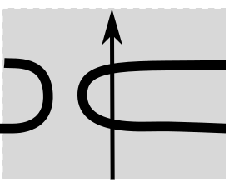}} \right) }
\def\Rfoura{\left(  \raisebox{-10pt}{\incl{1 cm}{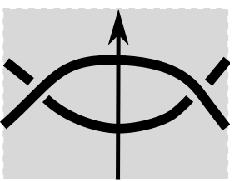}} \right) }
\def\Rfourb{\left(  \raisebox{-10pt}{\incl{1 cm}{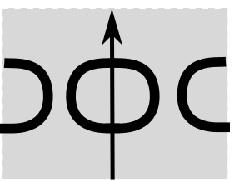}} \right) }
\def\pfRtta{\left(  \raisebox{-10pt}{\incl{1 cm}{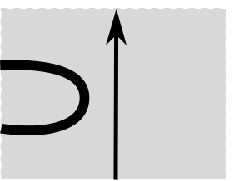}} \right) }
\def\pfRttb{\left(  \raisebox{-10pt}{\incl{1 cm}{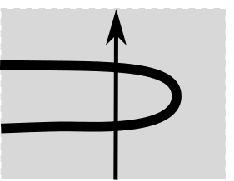}} \right) }

\def\pfRttone{\left(  \raisebox{-10pt}{\incl{1 cm}{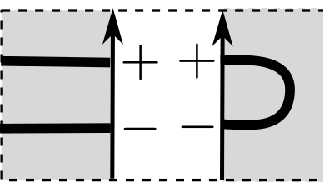}} \right) }
\def\pfRtttwo{\left(  \raisebox{-10pt}{\incl{1 cm}{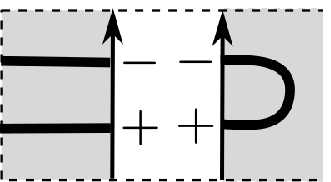}} \right) }
\def\pfRttthree{\left(  \raisebox{-10pt}{\incl{1 cm}{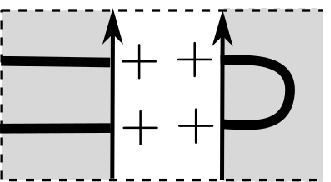}} \right) }
\def\pfRttfour{\left(  \raisebox{-10pt}{\incl{1 cm}{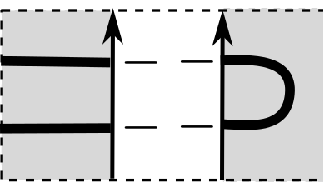}} \right) }
\def\pfRttd{\left(  \raisebox{-10pt}{\incl{1 cm}{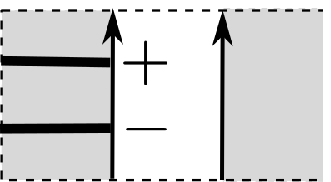}} \right) }
\def\pfRttdn{\left(  \raisebox{-10pt}{\incl{1 cm}{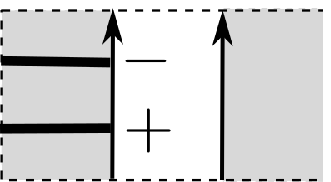}} \right) }
\def\pfRtte{\left(  \raisebox{-10pt}{\incl{1 cm}{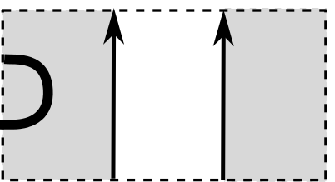}} \right) }
\def\ti{\tilde \rho}
\def\cross{  \raisebox{-8pt}{\incl{.8 cm}{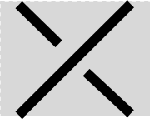}} }
\def\resoP{  \raisebox{-8pt}{\incl{.8 cm}{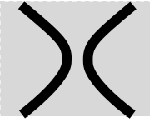}} }
\def\resoN{  \raisebox{-8pt}{\incl{.8 cm}{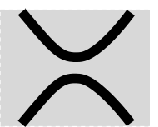}} }
\def\kinkp{  \raisebox{-8pt}{\incl{.8 cm}{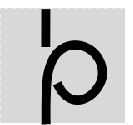}} }
\def\kinkn{  \raisebox{-8pt}{\incl{.8 cm}{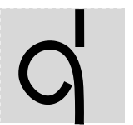}} }
\def\kinkzero{  \raisebox{-8pt}{\incl{.8 cm}{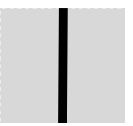}} }
\def\reordoneall{  \raisebox{-8pt}{\incl{.8 cm}{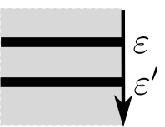}} }
\def\reordonepn{  \raisebox{-8pt}{\incl{.8 cm}{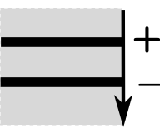}} }
\def\reordtwopn{  \raisebox{-8pt}{\incl{.8 cm}{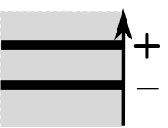}} }
\def\reordoneallp{  \raisebox{-8pt}{\incl{.8 cm}{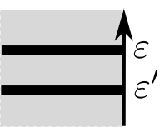}} }
\def\reordnp{  \raisebox{-8pt}{\incl{.8 cm}{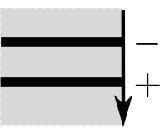}} }
\def\reordpn{  \raisebox{-8pt}{\incl{.8 cm}{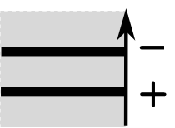}} }
\def\reordone{  \raisebox{-8pt}{\incl{.8 cm}{reord1}} }
\def\reordtwo{  \raisebox{-8pt}{\incl{.8 cm}{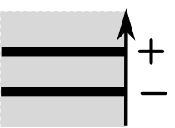}} }
\def\reordthree{  \raisebox{-8pt}{\incl{.8 cm}{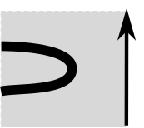}} }
\def\trivloop{  \raisebox{-8pt}{\incl{.8 cm}{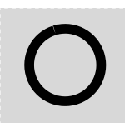}} }
\def\caseonea{  \raisebox{-10pt}{\incl{1 cm}{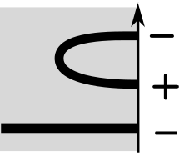}} }
\def\caseoneb{  \raisebox{-10pt}{\incl{1 cm}{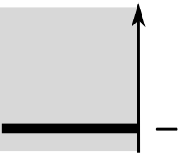}} }
\def\casetwoa{  \raisebox{-10pt}{\incl{1 cm}{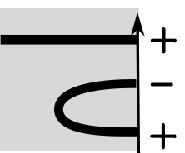}} }
\def\casetwob{  \raisebox{-10pt}{\incl{1 cm}{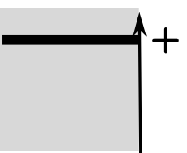}} }
\def\casethreea{  \raisebox{-10pt}{\incl{1 cm}{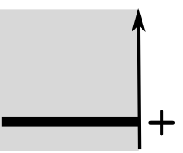}} }
\def\casethreeb{  \raisebox{-10pt}{\incl{1 cm}{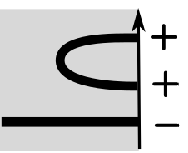}} }
\def\reordonez{  \raisebox{-8pt}{\incl{.8 cm}{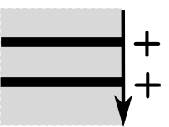}} }
\def\reordonea{  \raisebox{-8pt}{\incl{.8 cm}{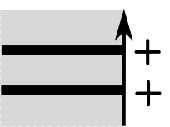}} }
\def\reordoneaa{  \raisebox{-8pt}{\incl{.8 cm}{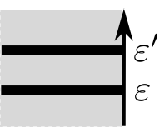}} }
\def\reordoneb{  \raisebox{-8pt}{\incl{.8 cm}{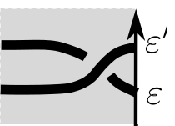}} }
\def\reordonec{  \raisebox{-8pt}{\incl{.8 cm}{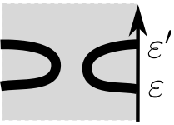}} }
\def\reordsixz{  \raisebox{-8pt}{\incl{.8 cm}{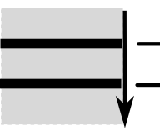}} }
\def\reordsixa{  \raisebox{-8pt}{\incl{.8 cm}{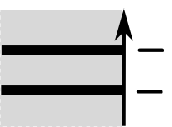}} }

\begin{document}

\title{Triangular decomposition of skein algebras}

\author[Thang  T. Q. L\^e]{Thang  T. Q. L\^e}
\address{School of Mathematics, 686 Cherry Street,
 Georgia Tech, Atlanta, GA 30332, USA}
\email{letu@math.gatech.edu}

\date{\today}

\thanks{Supported in part by National Science Foundation. \\
2010 {\em Mathematics Classification:} Primary 57N10. Secondary 57M25.}

\begin{abstract}
By introducing a finer version of the Kauffman bracket skein algebra, we show how to decompose the Kauffman bracket skein algebra of a surface into elementary blocks corresponding to the triangles in an ideal triangulation of the surface. The new skein algebra of an ideal triangle has a simple presentation. This gives an  easy proof of the existence of the quantum trace map of Bonahon and Wong. We also explain the relation between our skein algebra and the one defined by Muller, and use it to show that the quantum trace map can be extended to the Muller skein algebra.
\end{abstract}

\maketitle

\def\pbbS{\partial \bbS}
\def\bSP{(\bbS,\cP)}
\def\cA{\mathcal A}
\def\cB{\mathcal B}
\def\Si{\Sigma}
\def\fB{\mathfrak B}
\def\pr{\mathrm{pr}}
\def\cO{\mathcal O}

\def\poS{\partial_0\Sigma}
\def\cSs{\cS_{\mathrm s}}
\def\cR{\mathcal R}
\def\basics{basic skein}
\def\YD{\cY(\D)}
\def\tYD{\tilde \cY(\D)}
\def\YtD{\cY^{(2)}(\D)}
\def\tYtD{\tilde \cY^{(2)}(\D)}
\def\bve{{\boldsymbol{\ve}}}
\def\bm{{\mathbf m}}
\def\hYeD{\tilde \cY^\ev(\D)}
\def\YeD{\cY^\ev(\D)}
\def\pfS{\partial \fS}
\def\pbfS{\partial \bfS}
\def\bove{\boldsymbol{\ve}}
\def\bomu{{\boldsymbol \mu}}
\def\bonu{{\boldsymbol \nu}}
\def\Gr{\mathrm{Gr}}
\def\bl{\mathbf{l}}
\def\onto{\twoheadrightarrow}
\def\Stink{\mathrm{St}^\uparrow(\bk)}
\def\Stinkp{\mathrm{St}^\uparrow(\bk')}
\def\sincr{s^\uparrow}
\def\St{\mathrm{St}}

\def\tF{\tilde {\cF}}
\def\tE{\tilde {\cE}}
\def\tfT{\tilde {\fT}}
\def\bcSs{\overline{\cS}_s}
\def\cN{\mathcal N}
\def\MN{(\bfS,\cP)}
\def\SMuller{\cS^{\mathrm{Muller}}}
 \def\cSsp{\cS_{s,+}}
   \def\TrD{\Tr_\D}
   \def\hTrD{\widehat{\Tr}_\D}
   \def\hcSs{\widehat{\cS}_s}

\section{Introduction}

\subsection{Kauffman bracket skein algebra of  surface
}

 Throughout this paper $\cR$ is a commutative ring with unit $1$ and a distinguished invertible element $q^{1/2} \in \cR$.

 Suppose $\fS= \bfS\setminus \cP$, where $\bfS$ is a compact oriented 2-dimensional manifold with (possibly empty) boundary $\pbfS$ and $\cP$ is a finite set.
The {\em Kauffman bracket skein algebra} $\ooS(\fS)$, introduced by Przytycki~\cite{Prz} and Turaev~\cite{Turaev}, is defined as the $\cR$-module spanned by isotopy classes of framed unoriented links in $\fS \times (0,1)$ modulo the skein relation \eqref{eq.skein0} and the trivial loop relation \eqref{eq.loop0}:
\begin{align}
\lbl{eq.skein0} \cross \ &= \ q\resoP + q^{-1} \resoN\\
\lbl{eq.loop0}  \trivloop\  &=\  (-q^2 -q^{-2})\emptyr
 \end{align}
 For a detailed explanation of these formulas,
 as well as other formulas and notions in the introduction,
 see Section \ref{sec.def}. There  is a natural product making $\ooS(\fS)$  an $\cR$-algebra, which has played an important role in low-dimensional topology and quantum topology. In particular, it is known that $\ooS(\fS)$ is a  quantization of the $SL_2(\BC)$-character variety of the fundamental group of $\fS$ along the Weil-Petersson-Goldman bracket \cite{Turaev,Bullock,PS1,BFK}. The algebra $\ooS(\fS)$ and its cousin defined for 3-manifolds have helped to establish the AJ conjecture, relating the Jones polynomial and the A-polynomial of  a knot, for a certain class of knots \cite{Le:AJ,LZ}. A construction of Topological Quantum Field Theory is based on $\ooS(\fS)$~\cite{BHMV}. Recently, $\ooS(\fS)$ is found to have  relations with quantum cluster algebras and quantum \Teich spaces \cite{BW1}, and we also discuss these relations in this paper.

 The skein algebra $\ooS(\fS)$ is defined using geometric objects in a 3-manifold, and we want to understand its algebraic aspects.

 \def\tal{\tilde \al}
 \def\fT{\mathfrak T}
\subsection{Decomposition}

Assume that each connected component of the boundary $\pfS$ is diffeomorphic to the open interval $(0,1)$. Such a $\fS$ is called a {\em punctured bordered surface} in this paper.
Every connected component of $\pfS$ is called a {\em boundary edge} of $\fS$.

Very often $\fS$  has an {\em ideal triangulation}. This means,
$\fS$ can be obtained from a finite collection  of disjoint {\em ideal triangles} by gluing together  some pairs of edges of these triangles. Here an {\em ideal triangle} is a triangle without vertices. We want to know if one can build, or understand, the skein algebra of $\fS$ from those of the ideal triangles and the way they are glued together.  This is reduced to the question how the skein algebra behaves under gluing of boundary edges.

Suppose $a,b$ are distinct boundary edges of $\fS$. Let $\fS'$ be the result of gluing $a$ and $b$ together in such a way that the  orientation is compatible, ie  $\fS'= \fS/(a=b)$. We don't assume that $\fS$ is connected. It is clear that if we want to relate $\ooS(\fS')$ to $\ooS(\fS)$, we  have to enlarge the skein algebra to involve the boundary $\pfS$.

 By a {\em $\pfS$-tangle} we mean a compact, framed, one-dimensional proper submanifold $\al$ of $\fS \times (0,1)$ such that
 \begin{itemize}
 \item at every boundary point of $\al$ the framing is vertical, and
 \item for every boundary edge $e$, the points in
$\al \cap (e \times (0,1))$  have distinct heights,
 \end{itemize}
(see details in Section \ref{sec.def}). A {\em stated $\pfS$-tangle}  is a $\pfS$-tangle equipped with a map $s: \partial \al \to \{\pm \}$, called a state of $\al$.

We define the {\em stated skein algebra} $\cSs(\fS)$ to be the free $\cR$-module spanned by the isotopy classes of stated $\pfS$-tangles modulo the usual skein relation \eqref{eq.skein0}, the trivial loop relation \eqref{eq.loop0}, and the new boundary relations \eqref{eq.arcs0} and \eqref{eq.order0}, again see Section \ref{sec.def} for details.
\begin{align}
\lbl{eq.arcs0} \leftup\  & =\  q^{-1/2} \emptys\ , \qquad \leftupPP\ =0, \quad \  \leftupNN \ = 0   \\
 \lbl{eq.order0}   \reordone\ &=\  q^2 \reordtwo \ +\  q^{-1/2} \reordthree
 \end{align}

Let $\pr: \fS \onto \fS'$ be the natural projection, and $c = \pr(a)=\pr(b)$.
 Suppose $\al\subset (\fS' \times (0,1))$ is a stated $\pfS'$-tangle such that
 \begin{align}
 \lbl{eq.55}  &\text  {$\al$ is transversal to $c\times (0,1)$,}\\
\lbl{eq.55a}  &\text {the points in $\al \cap (c \times (0,1))$ have distinct heights and have vertical framing.}
 \end{align}
 Then  $\tal:=\pr^{-1}(\al) \subset \fS \times (0,1)$ is a $\pfS$-tangle and inherits states from $\al$ at all boundary points, except for those in $(a\cup b) \times (0,1)$.
For every $\bove: \al \cap (c \times (0,1) ) \to \{\pm\}$ let $\tal(\bove)$ be the stated $\pfS$-tangle whose states on $(a\cup b) \times (0,1)$ are the lift of $\bove$.

\begin{thm}\lbl{thm.I} Assume $\fS$ is punctured bordered surface, and  $\fS'= \fS/(a=b)$, where $a,b$ are boundary edges of  $\fS$. Let $\pr: \fS \onto \fS'$ be  the natural projection and $c=\pr(a)=\pr(b)$.

(a) There exists a unique $\cR$-algebra homomorphism $\rho :\cSs(\fS') \to \cSs(\fS)$ such that if  $\al$ is a
 $\pfS'$-tangle  satisfying \eqref{eq.55} and \eqref{eq.55a}, then
 $$ \rho(\al) = \sum_{\bove} \tal(\bove).$$
 Here the sum is over all maps $\bove: \al \cap (c \times (0,1))\to \{ \pm \}$.

 (b) In addition, $\rho$ is injective.

 (c) For any 4 distinct boundary edges $a_1,a_2,b_1,b_2$ of $\fS$, the following diagram is commutative:
\be
\lbl{eq.dia20}
 \begin{CD}   \cSs(\fS/(a_1=b_1,a_2=b_2))  @> \rho  >>  \cSs(\fS/(a_1=b_1)) \\
 @V
 \rho    VV  @V V \rho   V    \\
  \cSs(\fS/(a_2=b_2) )   @> \rho  >>    \cSs(\fS) .
\end{CD}
\ee
\end{thm}
Theorem \ref{thm.I} is proved in Section \ref{sec.decomposing}.

If we use only the skein relation \eqref{eq.skein0} and the trivial loop relation \eqref{eq.loop0} in the definition of $\cSs(\fS)$, then we get a bigger algebra $\hcSs(\fS)$, which was first introduced by Bonahon and Wong in their work \cite{BW1} on the quantum trace map. Relation \eqref{eq.arcs0} was also implicitly given in \cite{BW1}, but Relation \eqref{eq.order0} is new. It is this new relation  \eqref{eq.order0} which is responsible  for the existence of the decomposition map $\rho: \cSs(\fS') \to \cSs(\fS)$ of Theorem \ref{thm.I}.

We will show $\ooS(\fS)$ embeds naturally into $\cSs(\fS)$, which follows from the consistency of the defining relations  (see Theorem \ref{thm.basis} and Section \ref{sec.def}).
If we want the consistency and the well-definedness of decomposition map, then the coefficients on the right hand side of relations \eqref{eq.arcs0} and \eqref{eq.order0} are uniquely determined up to certain symmetries, see Section \ref{sec.uniq}. The uniqueness makes the the definition of our skein algebra more or less canonical.

Concerning the structure of $\cSs(\fS)$, we also have the following, whose proof is given in Section \ref{sec.zero}.
\begin{thm}
\lbl{thm.zero}
Suppose the ground ring $\cR$ is a domain, and $\fS$ is a punctured bordered surface.
Then $\cSs(\fS)$ is a domain, i.e. if  $xy=0$ and $x,y\in \cSs(\fS)$, then $x=0$ or $y=0$.
\end{thm}
When $\pfS=\emptyset$, one has $\cSs(\fS)=\ooS(\fS)$, the original skein algebra, and the above result had been known in this case, see \cite{PS2,CM,BW1,Muller}.

\def\htrD{\widehat{\tr_\D}}
\subsection{Triangular decomposition} Suppose $\fS$  has an ideal triangulation $\D$, ie  $\fS$ can be obtained from a finite collection $\tF=\tF(\D)$ of disjoint  ideal triangles by gluing together  some pairs of edges of these triangles.
Choose an order of the gluing operations and apply Theorem \ref{thm.I} repeatedly, then we get  an algebra embedding
\be
\lbl{eq.tri0}
\rho_\D: \cSs(\fS) \embed \bigotimes_{\fT \in \tF} \cSs(\fT).
\ee
 Parts (b) and (c) of Theorem \ref{thm.I} show that $\rho_\D$ is injective, and does not depend on the order of gluing. The map $\rho_\D$, called a {\em triangular decomposition of $\cSs(\fS)$}, can be described explicitly by a state sum formula.

It is natural now to study the stated skein algebra of an ideal triangle $\fT$ as every representation of $\cSs(\fT)$ gives us a representation of the stated skein algebra.
  In Theorem \ref{r.present.fT}  we give an explicit presentation of the stated skein algebra of an ideal triangle, which has $12$ generators with a simple set of relations.

\subsection{Application: quantum trace map} To each triangulation $\D$ of $\fS$ there is associated the Chekhov-Fock algebra $\YD$, which is built from the Chekhov-Fock algebra $\cY(\fT)$ of the ideal triangle. Actually $\YD$ is a subalgebra of $\bigotimes_{\fT\in \tF} \cY(\fT)$, and
is a version of the multupiplicative Chekhov-Fock algebra studied in \cite{Liu,BW1,Hiatt}. Bonahon and Wong constructed a remarkable algebra map $\TrD:\hcSs\to \YD$, called the {\em quantum trace map}, which when $q=1$, is the classical trace map expressing the $PSL_2$-trace of a curve on the surface in terms of the Thurston shear coordinates of the Teichm\"uller space. The existence of the quantum trace map had been conjectured in \cite{Fock,CF2}. The construction of Bonahon and Wong is based on difficult calculations.

\def\vkD{\varkappa_\D}
As an application of our triangular decomposition,  we will show that the quantum trace map of Bonahon and Wong can be easily constructed using the triangular decomposition \eqref{eq.tri0} as follows. First, using the explicit presentation of $\cSs(\fT)$,  we construct an algebra homomorphism $\phi: \cSs(\fT) \to \cY(\fT)$. Then define $\vkD$ as the composition
\be
\vkD: \cSs(\fS) \overset {\rho_\D} \longrightarrow \bigotimes _{\fT \in \tF(\D)} \cSs(\fT) \overset {\otimes \phi} \longrightarrow \bigotimes _{\fT \in \tF(\D)} \cY(\fT) .
\ee

\begin{thm}\lbl{thm.II}
The composition $\TrD: \hcSs(\fS) \to  \cSs(\fS) \overset {\vkD} \longrightarrow   \bigotimes_{\fT\in \tF} \cY(\fT)$ coincides with the quantum trace map of Bonahon and Wong.
\end{thm}
The proof is easy, and is given in Section \ref{sec.QT}.
In essence, we replace the difficult calculations in \cite{BW1} by explicit presentation of the  stated skein algebra of the ideal triangle.

For another approach to the quantum trace map using the Muller skein algebra see \cite{Le:QT}.

\subsection{Relation to Muller's skein algebra}  For a  {\em marked surface}, i.e. a pair $ (\bfS,\cP)$ where $\bfS$ is a compact oriented 2-dimensional manifold with (possibly empty) boundary $\pbfS$ and a finite set $\cP$ in the boundary $\pbfS$,
Muller \cite{Muller} defines the skein algebra $\SMuller(\bfS,\cP)$ using the tangles whose end points are in $\cP \times (0,1)$. See section \ref{sec.Muller} for details. The Muller skein algebra is closed related quantum cluster algebras of marked surfaces.

Let $\fS= \bfS \setminus (\cP\cup  \partial'(\bfS))$, where  $\partial'(\bfS)$ is the union of all connected components of $\partial(\bfS)$ not intersecting $\cP$. Then $\fS$ is a punctured bordered surface.
 In Section \ref{sec.Muller} we show that there is a natural  $\cR$-algebra isomorphism
    $$\Omega: \SMuller(\bfS, \cP) \overset{\cong }\longrightarrow \cSsp(\fS),$$
where $\cSsp(\fS)$ is the subalgebra of $\cSs(\fS)$ generated by stated $\pfS$-tangles whose states are $+$ only.
 \def\bvkD{\bar \varkappa_\D}
  Using the isomorphism $\Omega$, we can define the  quantum trace map on $\SMuller(\bfS,\cP)$
   $$ \bvkD: \SMuller(\bfS,\cP) \overset \Omega \longrightarrow  \cSsp(\fS)  \overset \vkD \longrightarrow  \YD$$
   for any triangulation $\D$ of $\fS$. The following is an extension of \cite[Proposition 29]{BW1}.
  \begin{thm} \lbl{thm.inj}
  The quantum trace map  $\bvkD:\SMuller(\bfS,\cP) \to \YD$ is injective.
  \end{thm}

\subsection{Plan of paper} In Section \ref{sec.def} we give a detailed definition of the stated skein algebra $\cSs(\fS)$ of a punctured bordered surface $\fS$,  its symmetry, filtrations, and grading. We prove Theorem \ref{thm.basis} describing a natural  $\cR$-basis of $\cSs(\fS)$. The proof is a standard application of the diamond lemma.
In Section \ref{sec.decomposing} we prove Theorem \ref{thm.I}. In Section \ref{sec:triangle} we give a presentation of $\cSs(\fS)$ when $\fS$ is an ideal bigon or an ideal triangle, and prove Theorem \ref{thm.zero}. We prove Theorem \ref{thm.II} about the quantum trace map in Section \ref{sec.QT}. In Section \ref{sec.Muller} we discuss the Muller skein algebra and prove Theorem \ref{thm.inj}.

\subsection{Acknowledgement} Much of this work is inspired by the work of Bonahon and Wong \cite{BW1,BW2} on the quantum trace map. The author would like to thank  G.~Masbaum and V.~Turaev for helpful discussions.  The author is partially supported by an NSF grant.

\section{Punctured bordered  surfaces and skein algebras} \lbl{sec.def}
\subsection{Notations}
Throughout the paper let
$\BZ$ be the set of integers, $\BN$ the set of non-negative integers, $\BC$ the set of complex numbers.
The ground ring $\cR$  is a commutative ring with unit 1,  containing a invertible element $q^{1/2}$.
 For a finite set $X$ we denote by $|X|$ the number of elements of $X$.

\def\fT{\mathfrak T}
In this section we fix a
 {\em punctured bordered surface} $\fS$,  i.e. a surface obtained by removing a finite set $\cP$
from a compact oriented surface $\bfS $ with
(possibly empty) boundary $\pbfS$, with the assumption that every connected component of the boundary $\pbfS$ has at least one point in $\cP$. We don't require $\bfS $ be to connected.

Let $\pfS= \pbfS \setminus \cP$. A connected component of $\pfS$ is called a {\em boundary edge} of $\fS$. Every boundary edge is diffeomorphic to the open interval $(0,1)$.

\subsection{Tangles and height order} The boundary of the 3-manifold $\fS \times (0,1)$ is $\pfS \times (0,1)$. For a point $(z,t)\in \fS \times (0,1)$, $t$ is called its {\em height}. A vector at
 $(z,t)$ is called {\em vertical} if it is parallel to the $(0, 1)$
factor and points in the direction of 1. A 1-dimensional submanifold $\al$ of $\fS \times (0,1)$ is {\em framed} if it is equipped with a {\em framing},  i.e.
 a continuous choice of a vector transverse to $\al$ at each
point of $\al$.

In this paper, a {\em $\pfS$-tangle} is
an unoriented,
 framed, compact,  properly embedded 1-dimensional submanifold $\al \subset \fS \times (0,1)$ such that:
 \begin{itemize}
 \item  at every point of $\partial \al=\al \cap (\pfS \times (0,1))$ the framing is {\em vertical}, and
\item  for any boundary edge $b$, the points of $\partial_b(\al):=\partial \al \cap (b \times (0,1))$  have distinct~heights.
 \end{itemize}

Two $\pfS$-tangles are {\em isotopic} if they are isotopic in the class of $\pfS$-tangles. The emptyset, by convention, is a $\pfS$-tangle which is isotopic only to itself.

For a $\pfS$-tangle $\al$ define a partial order on $\partial(\al)$ by: $x>y$ if $x$ and $y$
are in the same boundary edge and $x$ has greater height. If $x>y$ and there is no $z$ such that $x>z>y$, we say
$x$ and $y$ are {\em consecutive}.

\subsection{Tangle diagrams, boundary order, positive order}
As usual, $\pfS$-tangles are depicted by their diagrams on $\fS$, as follows.
Every $\pfS$-tangle is isotopic to one with vertical framing.
Suppose a vertically framed $\pfS$-tangle $\al$ is in general position with respect to
the standard projection $\pi : \fS \times (0,1) \to \fS$, i.e.  the restriction $\pi |_{\al}:\al \to \fS$ is an immersion with transversal double points as the only possible singularities and there are no double points on the boundary of $\fS$.
Then $D=\pi (\al)$, together with the over/underpassing information at every double point is called
 a {\em  $\pfS$-tangle  diagram}. {\em Isotopies} of (boundary ordered) $\pfS$-tangle diagrams are  ambient isotopies in $\fS$.

A $\pfS$-tangle diagram $D$ with a total order on each set $\partial_b(D):=D \cap b$, for all boundary edge $b$, is called a {\em boundary ordered $\pfS$-tangle diagram}. For example,
the partial order on $\partial(\al)$ induces such a  boundary order on $\partial D$.

Every boundary ordered $\pfS$-tangle diagram determines a unique isotopy
class of the $\pfS$-tangle, where the framing is vertical everywhere. When there is no confusion,
we identify a boundary ordered $\pfS$-tangle diagram with its isotopy class of $\pfS$-tangles.

\def\ori{{\mathfrak o}}
 Let $\ori$ be an {\em orientation} of $\pfS$,   which on a boundary edge may or may not be equal to the orientation inherited from $\fS$. The {\em $\ori$-order} of  a $\pfS$-tangle diagram $D$,  is the order in which points on $\partial_b(D)$ are increasing when going along the direction of $\ori$. It is clear that every $\pfS$-tangle, after an isotopy, can be presented by an {\em $\ori$-ordered $\pfS$-tangle diagram}, i.e. a $\pfS$-tangle diagram with $\ori$-order.

 If $\ori$ is  the orientation coming from $\fS$, the $\ori$-order is called the {\em positive order}. Every isotopy class of $\pfS$-tangles can be presented by a positively ordered  $\pfS$-tangle diagram.

\subsection{Framed Reidemeister moves}

Every isotopy class of $\pfS$-tangle can be presented by infinitely many boundary ordered  $\pfS$-tangle diagrams. Just like in the theory of framed links, two positively ordered $\pfS$-tangle diagrams  represent isotopic $\pfS$-tangles if and only if one can be obtained from the other by a sequence of moves, each is either an isotopy in $\fS$ or one of  the framed Reidemeister moves RI, RII, and RIII, described in Figure \ref{fig:Rmoves}.
\FIGc{Rmoves}{Framed Reidemeister moves RI, RII, and RIII}{1.5cm}

If we don't restrict to positive order, then two boundary ordered $\pfS$-tangle diagrams  represent isotopic $\pfS$-tangles if and only if one can be obtained from the other by a sequence of moves, each is either an isotopy in $\fS$, one of RI, RII, RIII, and the exchange move described in Figure \ref{fig:Emove}.
\FIGc{Emove}{Exchange move. Here the arrowed interval is a part of a boundary edge, and the order on that part is such that the point closer to the tip of the arrow is higher.
Besides, these two points are consecutive in the height order.
}{1.2cm}


\subsection{Stated skein module/algebra}

A {\em stated $\pfS$-tangle $\al$} is a $\pfS$-tangle $\al$ equipped with a {\em state}, which is a function
 $s : \partial \al \to  \{+,-\}$. Similarly, {\em a stated $\pfS$-tangle diagram $D$} is an $\pfS$-tangle diagram $D$ equipped with  a state $s: \partial D \to \{ \pm \}$.

The {\em (Kauffman bracket) stated skein module}  $\cSs(\fS)$ is  the $\cR$-module freely spanned
 by isotopy classes of stated $\pfS$-tangles modulo the {\em defining relations}, which are the  skein relation~\eqref{eq.skein}, the trivial loop relation~\eqref{eq.loop}, and the boundary relations~\eqref{eq.arcs} and~\eqref{eq.order}:
\begin{align}
\lbl{eq.skein} \cross \ &= \ q\resoP + q^{-1} \resoN\\
\lbl{eq.loop}  \trivloop\  &=\  (-q^2 -q^{-2})\emptyr\\
\lbl{eq.arcs} \leftup\  & =\  q^{-1/2} \emptys\ , \qquad \leftupPP\ =0, \quad \  \leftupNN \ = 0   \\
 \lbl{eq.order}   \reordone\ &=\  q^2 \reordtwo \ +\  q^{-1/2} \reordthree
 \end{align}
Here is the convention about pictures in these identities, as well as in other identities in this paper. Each  shaded part is a  part of $\fS$, with a stated $\pfS$-tangle diagram on it.
Each arrowed line is part of a boundary edge, and the order on that part is indicated by the arrow and the points on that part are consecutive in the height order. The order of other end points away from the picture can be arbitrary and are not determined by the arrows of the picture.

Relation \eqref{eq.skein} says that if 3 boundary ordered, stated $\pfS$-tangle diagrams $D_1, D_2, D_3$ are identical everywhere, except for a small disk in which $D_1,D_2, D_3$ are like in respectively the first, the second, and the third shaded areas, then $[D_1] = q[D_2] + q^{-1} [D_3]$ in the skein module $\cSs(\fS)$. Here $[D_i]$ is the isotopy class of the stated $\pfS$-tangle determined by $D_i$.
Other relations are interpreted similarly.

For two $\pfS$-tangles $\al_1$ and $\al_2$ the product $\al_1 \al_2$  is defined as the result of stacking $\al_1$ above $\al_2$. That is, first isotope $\al_1$ and $\al_2$ so that $\al_1 \subset \fS \times (1/2,1)$ and $\al_2 \subset \fS \times (0, 1/2)$. Then $\al_1 \al_2= \al_1 \cup \al_2$.
 It is easy to see that this gives rise to a well defined product and hence an $\cR$-algebra structure on $\cSs(\fS)$.

It is clear that if   $\fS_1$ and $\fS_2$ are two punctured bordered  surfaces, then
there is a natural isomorphism
\be
\cSs (\fS_1 \sqcup \fS_2) \cong \cSs (\fS_1) \ot_\cR
\cSs (\fS_2).
\ee

\begin{remark} \lbl{rem.hS}
If we don't impose the boundary relations \eqref{eq.arcs} and \eqref{eq.order}, then we get a bigger skein module $\hcSs(\fS)$, which was first introduced in  \cite{BW1}.
Of course $\cSs(\fS)$ is a quotient of $\hcSs(\fS)$. If $\partial \fS=\emptyset$, then $\cSs(\fS)=\hcSs(\fS)=\ooS(\fS)$.
\end{remark}

\begin{remark} \lbl{rem.hS2}
If $\fS$ is allowed to have a closed boundary component, then unless $q^2=1$,
the defining relations are not consistent and the  skein module $\cSs(\fS)$ is small.
\end{remark}

\def\starto{\overset \star \longrightarrow}

\subsection{Consequences of defining relations} Define $C^\ve_{\ve'}$ for $\ve, \ve'\in \{\pm \}$ by
\be
\lbl{eq.Cve}
C^+_+ = C^-_- = 0, \quad C^+_-= q^{-1/2}, \quad C^-_+ = -q^{-5/2}.
\ee
\begin{lemma} \lbl{r.arcs}
In $\cSs(\fS)$ one has
\begin{align}
\lbl{eq.kink} -q^{-3} \kinkp \ &= \ \kinkzero\  = \ -q^3\kinkn\\
\lbl{eq.arcs1} \leftve\   &=\  C^\ve _{\ve'}  \emptys\\
\lbl{eq.arcs2} \leftvedl \ & = \ \leftved\   =\  - q^{3} C^{\ve'} _{\ve}  \emptys
\end{align}
\end{lemma}
\begin{proof} Identity \eqref{eq.kink} follows from the skein relation and the trivial loop relation, see \cite{Kauffman}.

Except for $(\ve,\ve')=(-,+)$, \eqref{eq.arcs1} is a defining relation.
Applying \eqref{eq.order}, then \eqref{eq.loop}, then \eqref{eq.arcs}, 
$$ \leftupp\  =  \ q^2 \leftup + q^{-1/2}\loopup \ = \  q^2( q^{-1/2}) + q^{-1/2}(-q^2 - q^{-2})= -q^{-5/2}, $$
which proves the remaining case of \eqref{eq.arcs1}.

The first equality of \eqref{eq.arcs2} follows from a rotation by $\pi$. Using isotopy, we have
$$ \ \leftved\   =\ \leftvetwist = -q^3  \leftve\    =\  -q^3 C^\ve _{\ve'}  \emptys $$
where the 2nd and the 3rd identities follow from \eqref{eq.kink} and \eqref{eq.arcs1}.
\end{proof}

\begin{lemma}[Height exchange move]\lbl{r.refl}
 (a) One has
\begin{align}
\lbl{eq.reor1}
\reordonez \  = \ q^{-1} \left( \reordonea \right), \quad \reordsixz \  &= \ q^{-1} \left( \reordsixa \right), \quad \reordonepn = q \left( \reordtwopn \right)\\
\lbl{eq.reor2}
q^{\frac 32} \reordpn - q^{-\frac 32} \reordnp &= (q^2 -q^{-2}) \reordthree.
\end{align}

(b) Consequently, if $q=1$ or $q=-1$, then for all $\ve, \ve'\in \{\pm \}$,
\be
\lbl{eq.sign}
 \reordoneall =
 q \reordoneallp.
 \ee

\end{lemma}
\begin{proof} (a) Using isotopy, then skein relation~\eqref{eq.skein}, and then~\eqref{eq.arcs1}, we have
 \be
 \lbl{eq.9s}
 \reordoneall \ =\  \reordoneb = q^{-1} \left( \reordoneaa \right) + q \left( \reordonec \right) \ = \ q^{-1} \left( \reordoneaa \right)   +  q C^{\ve'}_\ve\left(  \reordthree \right) .
 \ee
When $\ve=\ve'$, $C^{\ve'}_\ve=0$, and \eqref{eq.9s} proves the first two  identities of \eqref{eq.reor1}.

Suppose $\ve=+, \ve'=-$. Using  \eqref{eq.9s}, then \eqref{eq.order} and \eqref{eq.loop}, we have
$$ \reordonepn = q^{-1} \left( \reordone \right)  - q^{-3/2} \left( \reordthree  \right)= q \left( \reordtwo \right),$$
proving the last identity of \eqref{eq.reor1}.
Now suppose $\ve= +, \ve'=-$. Rewrite \eqref{eq.order} in the form
\be
\lbl{eq.neg}
  \reordtwo \ = \  q^{-2}   \reordone\   \  - \  q^{-5/2} \reordthree
  \ee
Using \eqref{eq.neg} in \eqref{eq.9s}, we get \eqref{eq.reor2}.

(b) follows from (a).
\end{proof}
\begin{corollary}
If $q=1$, then $\cSs(\fS)$ is commutative.
\end{corollary}
\begin{proof}
When $q=1$, Identity \eqref{eq.sign} shows that the height order does not matter in $\cSs(\fS)$. Besides, the skein relation show that over-crossing is the same as under-crossing. Hence for any two $\pfS$-tangles $\al, \beta$, we have $\al \beta = \beta \al$.
\end{proof}
\begin{remark} In general, because of relation  \eqref{eq.sign}, $\cSs(\fS)$ is not commutative when $q=-1$. For example, when $\fS$ is an ideal triangle, $\cSs(\fS)$ is not commutative when $q=-1$. This should be contrasted with the case of the usual skein algebra   $\ooS(\fS)$, which is
  commutative and is canonically equal to the $SL_2(\BC)$ character variety of $\pi_1(\fS)$ if $\cR=\BC$ and $q=-1$ (assuming $\fS$ is connected), see \cite{Bullock,PS1}.
\end{remark}

\def\id{\mathrm{id}}
\def\ofS{\mathring{\fS}}
\def\tsigma{\chi}
\def\bdelta{\boldsymbol{\delta}}
 \def\cL{\mathcal L}
 \def\tchi{\tilde \chi}

\subsection{Reflection anti-involution}  \lbl{sec.refl}
\begin{proposition}   \lbl{r.reflection} Suppose $\cR=\BZ[q^{\pm 1/2}]$. There exists a unique $\BZ$-linear
anti-automorphism $\chi: \cSs(\fS) \to \cSs(\fS)$, such that $\chi(q^{1/2})= q^{-1/2}$   and $\tsigma(\al)=\bar \al$, where $\al$ is a stated $\pfS$-tangle, and $\bar \al$ the image of $\al$ under the reflection of $\fS\times (0,1)$, defined by $(z,t) \to (z, 1-t)$. Here $\chi$ is an anti-automorphism means
 for any $x,y \in \cSs(\fS)$ and $r\in \cR$,
$$
\tsigma(x+y)= \tsigma(x) + \tsigma(y), \quad \tsigma(xy) = \tsigma(y) \tsigma(x).$$
\end{proposition}
\begin{proof}
Since $\cSs(\fS)$ is spanned by stated $\pfS$-tangles, the uniqueness is clear.

Let $\cL$ be the free $\cR$-module with basis the set of isotopy classes of stated $\pfS$-tangles and $\tchi: \cL \to \cL$
be the $\BZ$-linear map defined by $\tchi(r\al)= \bar r \bar \al$, where for  $r\in \cR$, $\bar r$ is the image of
$r$ under the involution $q^{1/2} \to q^{-1/2}$.
Using the height exchange move (Lemma \ref{r.refl}),
one sees that $\tsigma$  respects all the defining relations, and hence descends to a map
$\chi: \cSs(\fS) \to \cSs(\fS)$. It is clear that $\chi$ is an anti-automorphism.  \end{proof}

Clearly $\tsigma^2=\id$. We call $\tsigma$ the reflection anti-involution.

\subsection{Basis of stated skein module} \lbl{sec.basis0}

{ A $\pfS$-tangle diagram $D$} is {\em simple} if it has neither double point nor
trivial component. Here a closed component of $D$ is {\em trivial} if it bounds a disk in $\fS$,
 and an arc component of $\al$ is trivial if it can be homotoped relative its boundary,
 in the complement of other components of $D$,  to a subset of a boundary edge.
 By convention, the empty set is considered as a simple stated $\pfS$-tangle diagram with 0 component.

We order the set $\{ \pm \}$ so that $+$ is greater than $-$.
 A state $s: \partial D \to \{\pm\}$ of a boundary ordered $\pfS$-tangle diagram $D$ is {\em increasing}
if for any  $x,y \in \partial D$ with $ x \ge y$, one has $s(x) \ge s(y)$.
Thus, in an increasing state, on any
boundary edge, the points with $+$ state are above all the points with $-$ state.

Let $B(\fS)$ be the set of  of all
isotopy classes of  increasingly stated, positively ordered, simple $\pfS$-tangle diagrams.

\def\inv{\mathrm{nd}}
\begin{theorem}\lbl{thm.basis}
As an $\cR$-module, $\cSs(\fS)$ is free with basis $B(\fS)$.

\end{theorem}
\def\tB{\tilde B}
\begin{proof} The proof uses the diamond lemma, in the form explained in \cite{SW}.
 For a set $X$ denote by $\cR X$ the free $\cR$-module with basis $X$,
with the convention $\cR X = \{ 0\}$ when $X=\emptyset$. In this proof, all the $\pfS$-tangle diagram is assumed
to have positive order.

Let $\tB$ be the set of all isotopy classes of stated  $\pfS$-tangle diagrams.
Then $B=B(\fS)$ is a subset of $\tB$.
Define a binary relation $\to$ on $\cR \tB$ as follows.  First assume
$D\in V$ and $E \in \cR \tB$.  We write $D \to E$ if $D$ is any element in $\tB$
presented by the left hand side of an identity in the defining relations \eqref{eq.skein}--\eqref{eq.order},
 and $E$ is the corresponding right hand side.

Now assume $E', E''\in \cR \tB$, with  $E' =\sum_{i=1} ^kc_i D_i$, where $D_i\in \tB$.
We write $E' \to E''$ if there is  an index $j\le k$ and $E\in \cR \tB$ with  $D_j \to E$,
such that $E''$ is obtained from $E'$ by replacing $D_j$ with $E$ in the sum $\sum_{i=1} ^kc_i D_i$.

Let $\starto$ be the reflexive and transitive relation on $\cR \tB$ generated by $\to$,
 i.e.  $E \starto E'$ if either  $E=E'$ or there are $E_1, E_2, \dots, E_k\in \cR \tB$
with $E_i \to E_{i+1}$ for all $i=1, \dots, k-1$ such that $E_1=E, E_k= E'$.
If $E\starto E'$, we say $E'$ is a {\em descendant} of $E$.

Let $\sim$ be the equivalence relation on $\cR V$ generated by $\to$. We will prove the following two lemmas in Subsection \ref{sec.phu}.

\begin{lemma}  \lbl{r.p1}
One has $\cR \tB /\sim\  = \ \cSs(\fS)$.
\end{lemma}
\begin{lemma} \lbl{r.p2}
The relation  $\to$ is

(i) {\em terminal},  i.e. there does not exist an infinite sequence $E_1 \to E_2 \to E_3 \to \dots$, and

(ii) {\em locally confluent on $\tB$}, i.e. if $D \to E_1$ and $D \to E_2$ for some $D\in \tB$,
then $E_1, E_2$ have a common descendent.
\end{lemma}

Since $\to$ is terminal and locally confluent on $\tB$,
\cite[Theorem 2.3]{SW} shows that $\tB_{irr}$, the subset of elements $D\in \tB$ for which
there is no $E\in \cR \tB$
such that $D \to E$, is a basis of $\cR \tB/\sim$, which is $\cSs(\fS)$ (by Lemma \ref{r.p1}).
 It remains to notice
that $\tB_{irr}=B$.
The theorem is proved. \end{proof}

\def\ts{\tilde s}
\def\pr{\mathrm{pr}}

\subsection{Proofs of Lemmas \ref{r.p1} and \ref{r.p2}} \lbl{sec.phu}
\begin{proof}[Proof of Lemma \ref{r.p1}]
Two stated $\pfS$-tangle diagrams define the same stated $\pfS$-tangle if and only if
one can be obtained from the other by a sequence of framed Reidemeister moves RI, RII and RIII
 moves of Figure \ref{fig:Rmoves}. Thus, $\cSs(\fS)= \cR \tB/ (rel)$,
 where $(rel)$ consists of the defining relations \eqref{eq.skein}--\eqref{eq.order}
and the the moves RI, RII, RIII.
 But the three moves RI, RII, and RIII are consequences of the skein relation \eqref{eq.skein}
 and the trivial knot relation \eqref{eq.loop} (see \cite{Kauffman}). Hence,
$\cSs(\fS)= \cR \tB/ (rel)= \cR \tB /\sim$.
\end{proof}

Suppose $s:\partial(\al) \to \{\pm\}$ is a state of a $\pfS$-tangle $\al$.
A pair $(x,y) \in \partial(\al)^2$ is {\em $s$-decreasing} if $x>y$ and $s(x)=-, s(y)=+$. Let $\inv(s)$ be the number of $s$-decreasing pairs. Then $\inv(s)=0$ if and only if $s$ is increasing.


\begin{proof}[Proof of Lemma \ref{r.p2}] (a)
 For $D\in \tB $, with state $s$, let $\deg(D)$ be the sum of 4 terms: two times the number of double points,
the number of components, the number of boundary points, and $\inv(D)$.
By checking each of the relations  \eqref{eq.skein}--\eqref{eq.order},
one sees that that if $D\in \tB $ and $D \to E\in \cR \tB $, then $E$
is a linear combination of elements $D_j \in \tB $ with $\deg(D_j) <\deg(D)$. Hence, By \cite[Theorem 2.2]{SW} the relation $\to$ is terminal.

(b) Suppose $D$ is a stated $\pfS$-tangle diagram. For now we don't consider $D$ up to isotopy.
A disk $d \subset \fS$ is called {\em $D$-applicable} if $D \cap d$ is the left hand side of one of the defining relations \eqref{eq.skein}--\eqref{eq.order}. In that case let $F_d(D)= \sum_j c_j D_j$ be the corresponding right hand side, so that $D \to F_d(D)$. Here $D_j=D$ outside $d$.

 Suppose $E= \sum c_i D_i$, where $0\neq c_i \in \cR$ and $D_i$ is a stated $\pfS$-tangle diagram for each $i$. A disk $d\subset \fS$ is said to be {\em $E$-applicable} if $d$ is $D_i$-applicable for each $i$. In that case, define $F_d(E) = \sum c_i F_d(D_i)$. Clearly $E \starto F_d(E)$.

If $d_1, d_2$ are two disjoint $D$-applicable disks, then $d_1$ is $F_{d_2}(D)$-applicable and $d_2$ is $F_{d_1}(D)$-applicable, and $F_{d_1}(F_{d_2}(D)) = F_{d_2}(F_{d_1}(D))$ is a common descendant of $F_{d_1}(D)$ and $D_{d_2}(D)$.

Now suppose $D\to E_1$ and $D\to E_2$. We have to show that $E_1$ and $E_2$ have a common descendant.
There are applicable $D$-disks $d_1$ and $d_2$ such that $E_1= F_{d_1}(D)$ and $E_2= F_{d_2}(D)$.
If $d_1$ and $d_2$ are disjoint, then $E_1$ and $E_2$ have a common descendant.
 It remains the case when $d_1 \cap d_2 \neq \emptyset$.

The {\em support} of a $D$-applicable disk $d$ is defined to be
\begin{itemize}
\item the double point for the case of \eqref{eq.skein},
\item the closed disk bounded by the loop for the case of \eqref{eq.loop}
\item the closed disk bounded the arc of $D$ and part of the boundary edge between the two end points of the arc in the case of \eqref{eq.arcs}, and
    \item the closed interval between the two boundary points on the boundary edge in the case of \eqref{eq.order}.
\end{itemize}
In doing the move $D\to F_d(D)$, we can assume that $d$ is a small neighborhood of its support.
Hence if the supports of $d_1$ and $d_2$ are disjoint, then we can assume that $d_1$ and $d_2$ are disjoint.
By inspecting the left hand sides of \eqref{eq.skein}--\eqref{eq.order} we see that there are only three cases when supports of $d_1$ and $d_2$ are not disjoint. These cases are described in Figures \ref{fig:case1}--\ref{fig:case3}.

\FIGc{case1}{Case 1: The shaded area on the right (resp. left) is $d_1$ (resp. $d_2$). The shaded area in the middle is $d_1 \cup d_2$}{1.4cm}
\FIGc{case2}{Case 2: The shaded area on the right (resp. left) is $d_1$ (resp. $d_2$).
}{1.4cm}
\FIGc{case3}{Case 3: The shaded area on the right (resp. left) is $d_1$ (resp. $d_2$).
}{1.4cm}

Note that the proof of \eqref{eq.arcs1} actually shows that
\be
\lbl{eq.arcs1a}
\leftupp\  \starto -q^{-5/2} \emptys
\ee

Case 1. From \eqref{eq.arcs} we have $E_1= F_{d_1}(D)= 0$. Using  \eqref{eq.order} then \eqref{eq.arcs1a},
$$ E_2= F_{d_2}(D)=  q^2 \, \caseonea \ + q^{-1/2} \, \caseoneb \quad \starto \quad - q^{-1/2} \, \caseoneb \ + \ q^{-1/2} \, \caseoneb =0. $$

Case 2.  From \eqref{eq.arcs} we have $E_1= F_{d_1}(D)= 0$. Using  \eqref{eq.order} then \eqref{eq.arcs1a},
$$ E_2= F_{d_2}(D)=  q^2 \, \casetwoa \ + q^{-1/2} \, \casetwob \quad \starto \quad - q^{-1/2} \, \casetwob \ + \ q^{-1/2} \, \casetwob =0. $$

Case 3.  From \eqref{eq.arcs} we have
$$E_1= F_{d_1}(D)=  q^{-1/2} \, \casethreea $$
Using  \eqref{eq.arcs}, we have
$$ E_2 = F_{d_2}(D)= q^2\,  \casethreeb \ + \ q^{-1/2} \, \casethreea \quad \starto \quad q^{-1/2} \, \casethreea \ = E_1.$$
In all three cases, $E_1$ is a common descendant of $E_1$ and $E_2$, completing the proof.
\end{proof}

 \subsection{More general boundary order}\lbl{sec.order} Let $\ori$ be an {\em orientation} of
$\pfS$. For a boundary edge $b$, we say $\ori$ is {\em positive on $b$} if it is equal to the orientation inherited from $\fS$,
otherwise it is called {\em negative} on $b$.
 Equation \eqref{eq.order} can be rewritten as \eqref{eq.neg},
which expresses a positive order term  as a sum of  a negative order term  and a term of lesser complexity.
Let  $B(\ori;\fS)$ be the set of
 of all isotopy classes of increasingly stated, $\ori$-ordered, simple $\pfS$-tangle diagrams.
 The proof of Theorem \ref{thm.basis} can be easily modified to give the following more general statement.

\begin{theorem}    \lbl{thm.basis1a}
 Suppose $\fS$ is a punctured bordered surface and $\ori$ is an orientation of $\pfS$.
Then $B(\ori;\fS)$ is an $\cR$-basis of $\cSs(\fS)$.
  \end{theorem}

\subsection{Filtration} \lbl{sec.filtration} Note that $|\partial(\al)|$ is even, for any $\pfS$-tangle $\al$.
For each non-negative integer $m$ let $F_m=F_m(\cSs(\fS))$ be the $\cR$-submodule of $\cSs(\fS)$
 spanned by all $\pfS$-tangles $\al$ such that $|\partial(\al)|\le  2m$. Clearly $F_m \subset F_{m+1}$ and $F_m F_k \subset F_{m+k}$. In other words, $\cSs(\fS)$ is a filtered algebra with the filtration $\{ F_m\}$. The associated graded algebra is denoted by $\Gr(\cSs(\fS))$, with $\Gr_m(\cSs(\fS))= F_m/F_{m-1}$ for $m\ge 1$ and $\Gr_0= F_0$. The following is a consequence of Theorem \ref{thm.basis1a}.
\begin{proposition}\lbl{r.basis2}  Le $\fS$ be a punctured bordered surface and $\ori$ be an orientation of $\pfS$.\\
 (a) The set $\{ \al \in B(\ori;\fS) \mid |\partial(\al)|\le 2m\} $
is an $\cR$-basis of $F_m(\cSs(\fS))$.\\
(b) The set $\{ \al \in B(\ori;\fS) \mid |\partial(\al)|= 2m\} $ is an $\cR$-basis of $\Gr_m(\cSs(\fS))$.
\end{proposition}
\def\Ed{{\cE_\partial}}
\def\ld{\mathrm{lt}}
  \subsection{Grading}  \lbl{sec.grading} For a boundary edge $b$ and a stated  $\pfS$-tangle $\al$ define
  $$ \bdelta_\al(b) =   \sum_{u \in \partial_{b}(\al)} s(u) \in \BZ,$$
where, as usual, we identify $+$ with $+1$ and $-$ with $-1$. Let $\cE_\partial$ be the set of all boundary edges.
Then $\bdelta_\al \in \BZ^{\Ed}$, the set of all maps $\Ed\to \BZ$.

For $\bk\in \BZ^\Ed$ let $G_\bk= G_\bk(\cSs(\fS))$ be the $\cR$-submodule of $\cSs(\fS)$ spanned by all
stated $\pfS$-tangles $\al$ such that $\bdelta(\al)=\bk$.
From the defining relations it is clear that $\cSs(\fS)= \bigoplus_{\bk \in \BZ^l} G_\bk$ and $G_\bk G_{\bk'} \subset G_{\bk + \bk'}$. In other words,  $\cSs(\fS)$ is a graded algebra with the grading $\{ G_\bk, k\}$.

Fix a boundary edge $b$. For $k\in \BZ$ let  $G_{b,k}= G_{b,k}(\cSs(\fS))$ be the $\cR$-submodule of $\cSs(\fS)$ spanned by all
stated $\pfS$-tangles $\al$ such that $\bdelta_\al(b)=k$.  Again $\cSs(\fS)= \bigoplus_{\ \in \BZ} G_{b,k}$ and $G_{b,k} G_{b,k'} \subset G_{b,k+k'}$. In other words,  $\cSs(\fS)$ is a $\BZ$-graded algebra with the grading $\{ G_{b,k}\}$.
  The following is a consequence of Theorem \ref{thm.basis1a}.
\begin{proposition}\lbl{r.basis3}
Le $\fS$ be a punctured bordered surface and $\ori$ be an orientation of $\pfS$.\\
 (a) The set $\{ \al \in B(\ori;\fS) \mid \bdelta_\al=\bk\} $ is an $\cR$-basis of $G_\bk(\cSs(\fS))$.\\
 (b) The set $\{ \al \in B(\ori;\fS) \mid \bdelta_\al(b)=k\} $ is an $\cR$-basis of $G_{b,k}(\cSs(\fS))$.

\end{proposition}
The $\BZ$-grading    $\{ G_{b, k}\}$ allows to define the {\em  $b$-leading term}
$\ld_b(x)$ of non-zero $ x \in \cSs(\fS)$. Suppose $x = \sum_j c_j D_j$, where $0\neq c_j \in \cR$ and
$D_j\in B(\ori;\fS)$. Assume $k= \max_j \bdelta_{D_j}(b)$. Define
\be
\ld_b(x)= \sum_{\bdelta_{D_j}(b) =k} c_j D_j.
\ee
\subsection{The ordinary skein algebra $\ooS(\fS)$}  Recall that the ordinary skein algebra $\ooS(\fS)$ is the $\cR$-module freely spanned by isotopy classes of framed links in $\fS \times (0,1)$ modulo the skein relation \eqref{eq.skein} and the trivial loop relation \eqref{eq.loop}.
The map $\id_*: \ooS(\fS) \to \cSs(\fS)$, defined on a framed link $\al$ by $\id_*(\al)= \al$, is an $\cR$-algebra homomorphism.
\begin{corollary} The map  $\id_*: \ooS(\fS) \to \cSs(\fS)$ is an embedding, and $\id_*(\ooS(\fS))= F_0(\cSs(\fS))$.
\end{corollary}
\begin{proof} The manifold $\ofS:=\fS \setminus \pfS$ is also a punctured bordered surface without boundary. It is clear that $\ooS(\fS) = \cSs(\ofS)$, and the $\cR$-basis of the latter described by Theorem \ref{thm.basis} is the $\cR$-basis of $F_0(\cSs(\fS))$. The result follows.
\end{proof}
\section{Decomposing and gluing punctured bordered  surfaces}
\lbl{sec.decomposing}

\def\tpr{\widetilde{\pr}}
\subsection{Gluing punctured bordered surfaces}    \lbl{sec.31}
Suppose $a$ and $b$ are distinct boundary edges of a punctured bordered surface $\fS$. Let $\fS'= \fS/(a=b)$, the result of gluing $a$ and $b$ together in such a way that the  orientation is compatible. The canonical projection $\pr: \fS \to \fS'$  induces a projection $\tpr: \fS \times (0,1) \to \fS'\times (0,1)$.
Let $c= \pr(a)=\pr(b)$. 

 A $\pfS'$-tangle  $\al\subset (\fS' \times (0,1))$,  is said to be {\em vertically transverse to $c$} if
 \begin{itemize}
 \item $\al$ is transversal to $c \times (0,1)$,
 \item the points in $\partial_c\, \al:= \al \cap (c \times (0,1))$ have distinct heights, and have vertical framing.
 \end{itemize}
 Suppose $\al$ is a $\pfS'$-tangle vertically transversal to $c$. Then $\tal:=\tpr^{-1}(\al)$ is a $\pfS$-tangle. Suppose in addition $\al$ is stated, with state $s: \partial \al \to \{\pm\}$. For any $\bove: \al \cap (c \times (0,1)) \to \{ \pm \}$ define $\tal(\bove)$ be $\tal$ equipped with state $\tilde s$ defined by $\tilde s (x) = s(\pr(x))$ if $\pr(x) \in \partial \al$ and $\tilde s (x) = \bove(\pr(x))$ if $\pr(x) \in c$. We call $\tal(\bove)$  {\em a lift} of $\al$. If $|\al \cap (c \times (0,1))|=k$, then  $\al$ has $2^k$ lifts.

\subsection{Proof of Theorem \ref{thm.I}} For the reader convenience
we reformulate Theorem \ref{thm.I} here.

\begin{theorem} \lbl{thm.1a}
Suppose $a$ and $b$ are two distinct boundary edges of a punctured bordered surface $\fS$. Let $\fS'=\fS/(a=b)$, and $c$ be the image of $a$ (or $b)$ in $\fS'$.

(a)
There is a unique $\cR$-algebra homomorphism
$ \rho : \cSs(\fS') \to \cSs(\fS)$
such that if $\al$ is a stated $\pfS'$-tangle  vertically transversal to $c$, then
$\rho (\al)=\sum_\beta [\beta]$, where the sum is over all lifts $\beta$ of $\al$, and $[\beta]$ is the element in $\cSs(\fS)$ represented by $\beta$.
\no{
$$ \rho (\al) = \sum_{\beta \in  \{ \text{lifts of } \al \} } \beta.$$
}

(b) In addition, $\rho $ is injective.

(c) For 4 distinct boundary edges $a_1,a_2,b_1,b_2$ of $\fS$, the following diagram is commutative:
\be
\lbl{eq.dia2}
 \begin{CD} \cSs(\fS/(a_1=b_1,a_2=b_2))  @> \rho  >>  \cSs(\fS/(a_1=b_1))\\
 @V
 \rho    VV  @V V \rho   V    \\
  \cSs((\fS/(a_2=b_2)    @> \rho  >>    \cSs(\fS) .
\end{CD}
\ee
\end{theorem}

\def\id{\mathrm{id}}
\def\cD{\mathcal D}
\def\cDo{\mathcal D^{\mathrm {ord}}}
\def\iso{\mathrm{iso}}
\begin{proof} (a) 
Let $T(c)$ be the set of all stated $\pfS'$-tangles vertically transverse to $c$
 (no isotopy is considered here), and $V$ be the set of all isotopy classes of stated $\pfS'$-tangles.
 The  map $\iso: T(c) \to V$, sending an element in $T(c)$ to its isotopy class as a stated $\pfS'$-tangle,
 is surjective.
Define
$$ \ti: T(c) \to \cSs(\fS), \quad \ti(\al) = \sum_{\beta: \ \text{lifts of }  \al} [\beta].$$

{\bf Claim 1.}  If $\al, \al'\in T(c) $ and $\iso(\al)=\iso(\al')$, then $\ti (\al)= \ti (\al')$.

Suppose the claim holds. Then $\ti$ descends to a map $\rho ': V \to \cSs(\fS)$.
Recall $\cSs(\fS')$ is defined as the $\cR$-span of $V$ modulo the defining relations \eqref{eq.skein}--\eqref{eq.order}.
The locality (of these defining relations) shows that $\rho'$ respects the defining relations. Hence $\rho'$ descends to an $\cR$-homomorphism  $\rho : \cSs(\fS') \to \cSs(\fS)$, which is clearly an $\cR$-algebra homomorphism.

It remains to prove Claim 1. We break the proof into steps.

\def\tD{\tilde D}

{\bf Step 1.} Let $\cD(c)$ be the set of all      
stated $\pfS'$-tangle diagrams transversal to $c$. Each $D\in \cD(c)$ is equipped with the positive boundary order. For a total order $\cO$ on $D \cap c$ and a map $\bove: D \cap c \to \{\pm \}$ let $\tD(\cO,\bove)$ be the stated  $\pfS$-tangle diagram obtained from $D$ by splitting along $c$. Here the height order and the states on $a$ and $b$ are the lifts of $\cO$ and $\bove$, while the height order and the states on other boundary edges are the lifts of the corresponding ones of $D$.  Define
$$ \ti(D,\cO) = \sum_{\bove} \tD(\cO,\bove), \quad \text{sum is over all maps $\bove: D \cap c \to \{\pm \}.$ }$$

{\bf Step 2.} Recall that $\al\in T(c)$, i.e. $\al$ is a stated $\pfS'$-tangle vertically transversal to $c$. A small smooth isotopy, keeping framing vertical on $c \times (0,1)$, does not move $\al$ out of $T(c)$, and does not change $\ti(\al)$.
Thus, after a small smooth isotopy of this type we can assume that
 $\al$ has  a stated $\pfS'$-tangle diagram $D\in \cD(c)$.  The height order on $\al \cap c \times (0,1)$ induces  a total order $\cO$ on $D \cap c$.  From the definition, $\ti(\al) = \ti(D,\cO)$.

 Similarly, $\al'$ is presented by $(D', \cO')$, where $D'\in \cD(c)$ and $\cO'$ is a total order on $D'\cap c$. To prove the claim, we need to show that $\ti(D,\cO)=\ti(D', \cO')$.

{\bf Step 3.} Recall that $\al$ and $\al'$ are isotopic and their diagrams $D,D'$ are transersal to $c$.
By considering
Reidemeister moves involving $D \cup c$, we see that
 $D'\cup c $ can be obtained from $D\cup c$ by a sequence of moves, each is \\
\noindent (i) a Reidemeister move RI, RII, or RIII not involving $c$; or \\
(ii) move IIa  (which involves $c$) as shown in Figure \ref{fig:R23a}; or \\
(iii)  move IIIa  (which involves $c$) as shown in Figure \ref{fig:R23a}; or \\
(iv) move IV which reorders   the total order on $D \cap c$, see Figure \ref{fig:R23a};\\
or an isotopy of $\fS'$ which fixes $c$ as a set during the isotopy.

\FIGc{R23a}{Move IIa (left), Move IIIa (middle), and Move IV (right). The vertical line is part of $c$, and the arrow indicates the order.}{1.5cm}
 It is clear that $\ti$ is invariant under isotopy of $\fS'$ which
 fixes $c$ and moves RI, RII, and RIII not involving $c$. We will see that if $\ti$ is invariant under move IIa, then it
is invariant under all other moves.

{\bf Step 4.}
Now we show that $\ti$ is invariant under move IIa.
We have
\begin{align*}
 \ti \pfRttb
&= \pfRttthree + \pfRttone + \pfRtttwo + \pfRttfour\\
&= -q^{5/2} \pfRttd + q^{1/2}\pfRttdn \\
&= -q^{5/2} \pfRttd + q^{1/2} \left   ( q^2 \pfRttd + q^{-1/2} \pfRtte \right)\\
&= \pfRtte = \ti \pfRtta.
\end{align*}
where the second identity follows from the values of trivial arcs given by \eqref{eq.arcs} and Lemma~\ref{r.arcs}, while the  third identity follows from \eqref{eq.order}.
Thus, $\ti$ is invariant under Move IIa.

{\bf Step 5.} Now we show that the invariance of Moves IIIa and IV follows from the invariance of Moves RI, RII, RIII (not involving $c$), and IIa.
Consider IIIa. Using the skein relation, we have
\begin{align}
\lbl{eq.11a} \ti \RthreeD &= q \, \ti \RthreeDpositive + q^{-1} \, \ti \RthreeDnegative   \\
\lbl{eq.11b} \ti \RthreeDp &= q \, \ti \RthreeDpositive + q^{-1} \, \ti \RthreeDpnegative.
\end{align}
The right hand sides of \eqref{eq.11a} and \eqref{eq.11b} are equal by move IIa, proving the invariance of IIIa.

Finally consider Move IV. Using the skein relation at the two crossings, we have
\begin{align*}
 \ti \Rfoura &=  \ti \RthreeDpositive +   \ti \Rfourb
q^{-2} \, \ti \RthreeDnegative +
  q^2 \,  \ti \RthreeDpnegative.
\end{align*}
Using Move IIa in the last 3 terms of the right hand side, then the trivial loop relation which says a trivial knot is $-q^2-q^{-2}$, we get
\begin{align*}
\ti \Rfoura &=  \ti \RthreeDpositive,
\end{align*}
which proves that $\ti$ is invariant under Move IV. This completes the proof of part (a).

\def\tB{\tilde B}
\def\ld{\mathrm{lt}}
\def\pp{\mathrm{qr}}
(b) Fix an orientation $\ori'$ of $\pfS'$ and an orientation of $c$. Define the orientation  $\ori$ of $\pfS$ such that the map $\pr: \fS\to \fS'$ preserve the orientation on edge boundary edge of $\fS$.
We will equip any  $\pfS'$-tangle diagram (resp. $\pfS'$-tangle diagram)  with the $\ori'$-order (resp. $\ori$-order).


Suppose $D$ is an increasingly stated $\pfS'$-tangle diagram transversal to $c$. Let $\tD(+)$ be the lift of $D$ in which all the state of every endpoint in $a$ (and hence in $b$) is $+$. Note that the state of $\tD(+)$ is also  increasing.
In general, $\tD$ may not be simple. After an isotopy we can assume that  $D$ is  {\em $c$-normal}, i.e.  $|D \cap c|=\mu(D,c)$, which is the smallest integer among all $|D'\cap c|$ with $D'$ isotopic to $D$. Then  $\tD=\pr^{-1}(D)$ is  a simple $\pfS$-tangle diagram, and  from the definition of $\rho$ and the $b$-leading term (see Section \ref{sec.grading}), we have
\be
\lbl{eq.lt1}
\ld_a(\rho(D)) = \tD(+) \quad \text{in \ }\  \cSs(\fS).
\ee
In particular, the isotopy class of  $\tD(+)$ does not depend on how we isotope $D$ to a $c$-normal  position. Note that the isotopy class of $\tD(+)$ totally determine the isotopy class of $D$, i.e. the following  map   is injective:
\be
\lbl{eq.liftp}
B(\ori';\fS')\to B(\ori;\fS), \quad  D\to \tD(+). \ee

Suppose $0\neq x\in \cSs(\fS')$. Then $x = \sum_j c_j D_j$, where $0\neq c_j \in \cR$ and  $D_j\in B(\ori';\fS')$. Assume $\max_j \mu(D_j, c)=k$.   From \eqref{eq.lt1} and the injectivity of the map \eqref{eq.liftp}, we have
\be
\lbl{eq.lt2}
\ld_a(\rho(x)) = \sum_{\mu(D_j, c)=k} c_j \tD_j(+) \neq 0.
\ee
 This proves $\rho(x)\neq 0$, and $\rho$ is injective.

(c) The commutativity of Diagram \eqref{eq.dia2} follows immediately from the definition.
\end{proof}
\subsection{Triangular decomposition} \lbl{sec.tri}  Suppose a punctured bordered surface $\fS$ is obtained by removing a finite set $\cP$ from a compact oriented surface $\bfS$.

Suppose $\D$ is an {\em ideal triangulation} of  $\fS$, i.e.   a triangulation of $\bfS$ whose vertex set is exactly $\cP$. 
By cutting along all the edges of $\D$, we see that
 there is a finite collection $\tF=\tF(\D)$ of disjoint ideal triangles and a finite collection of disjoint pairs of elements in $\tE=\tE(\D)$, the set of all edges of ideal triangles in $\tF$, such that $\fS$ is obtained from $\tS:=\bigsqcup_{\fT \in \tF} \fT$ by gluing the two edges in each pair.  It may happen that two edges of one triangle are glued together.

From  Theorem \ref{thm.1a} we have  an injective algebra homomorphism
\be
\lbl{eq.tri}
\rho_\D: \cSs(\fS) \embed \bigotimes_{\tau \in \cF(\D)} \cSs(\fT).
\ee
 The map $\rho_\D$ is  described explicitly by Theorem \ref{thm.1a}. It is natural now to study the stated skein algebra of an ideal triangle.

     It is known that $\fS$ is {\em triangulable}, i.e. it has a triangulation, if and only if
$|\cP|\ge 1$  and $(\bfS, \cP)$ is not one of the followings: (i) $\bfS$ is a sphere with $|\cP|\le 2$, (ii)  $\bfS$ is a disk with $\cP\subset\pS$ and $|\cP| \le 2$.

\subsection{On the uniqueness of the defining relations}
\lbl{sec.uniq}
Suppose we modify the defining relations by replacing \eqref{eq.arcs} and \eqref{eq.order} with respectively the more general
\begin{align}
\lbl{a1}
\leftup\  & =\  z_1 \emptys\ , \qquad \leftupPP\ =z_2 \emptys\ , \quad \  \leftupNN \ = z_3 \emptys\   \\
\lbl{a2}  \reordone\ &=\  z_4 \reordtwo \ +\  z_5 \reordthree,
\end{align}
where $z_i\in \cR$. Then it is easy the set $B(\fS)$, still spans the new $\cSs(\fS)$. If we want (i) consistency: $B(\fS)$ is a basis of $\cSs(\fS)$ and (ii) decomposition: Theorem \ref{thm.1a} holds, then repeating the proofs we can find  exactly 4 solutions $(z_1, z_2, z_3, z_4, z_5)$. In all of them $z_2=z_3=0$. The four solutions are
\begin{align*}
z_1&=z_5= \ve q^{-1/2}, z_4= q^2, \quad \ve \in \{ \pm 1 \}\\
z_1&=z_5= \ve q^{-5/2}, z_4= q^{-2}, \quad \ve \in \{ \pm 1 \}.
\end{align*}

 The group $\BZ/2 \times \BZ/2$, generated by two commuting involutions, acts on the set of solutions as follows. The first involution replaces each diagram $\al$ with $k$ $+$ states in \eqref{a1} and \eqref{a2} by $(-1)^k \al$. The second involution  switches all the states from $-$ to $+$ and $+$ to $-$ in \eqref{a1} and \eqref{a2}. It is easy to see that each involution transforms a solution to another solution. Then all the 4 solutions are obtained from one of them, say the solution we used in \eqref{eq.arcs} and \eqref{eq.order}, by the action of this group $\BZ/2 \times \BZ/2$. In this sense our solution is unique.

\def\rep{{\mathrm{rep}}}

\begin{remark}
Using Theorem \ref{thm.1a} one can interpret the assignment $\fS\to \cSs(\fS)$ as a co-presentation of a certain modular operad of surfaces. For details of modular operads, see~\cite{Voronov}.
\end{remark}
\no{
\subsection{Operadic interpretation}
\lbl{sec.operad}

A punctured bordered surface is {\em labeled} if
 the boundary edges are numbered from 1 to $n$, the number of boundary edges. Two labeled bordered surface with punctures are {\em equivalent} if there is an orientation and labels preserving diffeomorphism from the first to the second.

 In each equivalence class of labeled, connected, punctured bordered surfaces fix a representative and denote by $\cO$ the set of all representative. For a collection of representatives $\fS_1,\dots,\fS_k$
 let $\fS_1 \otimes  \dots \otimes \fS_k$ be the labeled punctured surface which is the disjoint union of all the $\fS_i$, with the boundary edges renumbered so that they  begin with $\fS_1$, then move to $fS_2$, $\dots$, and eventually to $\fS_k$, and for each $i$, the numbering of boundary edges of $\fS_i$ is a shift of the original numbering.

 Suppose $\fS$ is a labeled punctured bordered surface. A marking of , and $\fS_{\rep}$ is the representative of its equivalence class. {\em A marking of $\fS$} is a labels preserving diffeomorphism $f: \fS \to \fS_{\rep}$, considered up to isotopy of $\fS_\rep$.

 Let $\cO(n)$ be the set of pairs $(\fS,f)$ consisting of a  labeled punctured bordered surface $n$ boundary edges and a marking $f$. Suppose $\fS_1\in \cO(n), \fS_2\in \cO_m$. Let $\fS_1 \otimes \fS_2$

 Define the following operations
\begin{align}
\otimes:& \cO(n) \times \cO(m) \to \cO(n+m) \\
G_{i,j}:& \cO(n) \times \cO(m) \to \cO(n+m-2), \quad 1\le i\le n, \ 1 \le j \le m \\
G_{i,j}: &\cO(n)  \to \cO(n-2) \quad 1\le i < j \le n
\end{align} by
\begin{align}
S_1 \otimes S_2 & = S_1 \sqcup S_2, 
\\
G_{i,j}(S_1, S_2) &= S_1 \sqcup S_2 /(b_{i,1} = b_{j,2})  \\
G_{i,j}(S)&= S/(b_{i} = b_{j}).
\end{align}
(with natural labelings).
Modular operad, \cite{Voronov}.

Then the assignment $\fS\to \cSs(\fS)$ is a co-representation of the above operad.
}

\def\fB{\mathfrak B}
\def\pfB{\partial \mathfrak B}
\def\bequal{\overset \bullet =}
\section{Ideal bigon and ideal triangle} \lbl{sec:triangle}

\subsection{Definitions and notations}

An {\em arc} $\al$  in a punctured bordered surface $\fS$ is
a properly embedded submanifold diffeomorphic to $[0,1]$.
If the two end points of $\al$ are in the same boundary edge, we call $\al$ a {\em returning arcs}.

Suppose $s:\partial(\al) \to \{\pm\}$ is a state of a $\pfS$-tangle $\al$, where $\fS$ is a punctured bordered surface. A {\em permutation} of $s$ is any state of the form $s \circ \sigma$, where $\sigma: \partial(\al) \to \partial (\al)$ is a bijection such that if $x\in b$, where $b$ is a boundary edge, then $\sigma(x) \in b$. The only permutation of $s$ which is increasing is denoted by $\sincr$.

Suppose $s:\partial(\al) \to \{\pm\}$ is not increasing. There there is a pair $u,v\in \partial (\al)$ such that $u>v$, $u$ and $v$  are consecutive in the height order,  and $s(u)=-, s(v)=+$. The new state $s':\partial (\al)\to \{\pm \}$, which is equal to $s$ everywhere except $s'(u)=+, s'(v)=-$, is called a {\em simple positive permutation } of $s$.

For elements $x,y$ of an $\cR$-module, $x \bequal y$ will mean there is an integer $j$ such that $ x = q^{j/2} y$.

\subsection{Ideal bigon}
\FIGc{bigon}{From left to right: bigon, arc $\al$, $\al(2)$, $\al^2$, and $\al(-, +)$}{3cm}

Suppose $\fB$ is an {\em ideal bigon}, i.e. $\fB$ is obtained from a disk by removing 2 points on its boundary. Let $a$ and $b$ be the boundary edges of $\fB$. Let $\al$ be an arc whose two end points are not in the same boundary edge, and let $\al(k)$ be $k$ parallels of $\al$. See Figure \ref{fig:bigon}.
Unless otherwise stated, the order of each $\pfB$-tangle diagram is positive. For example, the diagram of $\al^2$ is different from  $\al(2)$ and is depicted in Figure \ref{fig:bigon}. For $\ve, \ve'\in \{\pm\}$ let $\al(\ve,\ve')$ be $\al$ equipped with the state $s$ such that $s(\al\cap a)=\ve, s (\al\cap b)= \ve')$. Recall that $C^\ve_{\ve'}$ is defined by \eqref{eq.Cve}.

\begin{theorem} \lbl{r.present.fB}
Let $\fB$ be an ideal bigon with the above notations. Then
$\cSs(\fB)$ is the $\cR$-algebra generated by $X =\{ \al(\ve, \ve') \mid \ve, \ve' \in \{ \pm \}\}$, subject to the relations
\begin{align}
\lbl{re1}
\al(\ve,-) \al(\ve',+) &= q^2 \al(\ve,+) \al(\ve', -) - q^{5/2} C^\ve_{\ve'}  \quad \forall \ve, \ve'
\in \{ \pm \} \\
\lbl{re2}
\al(-, \ve) \al(+,\ve') &= q^2 \al(+, \ve) \al(-,\ve') - q^{5/2} C^\ve_{\ve'} \quad \forall \ve, \ve'
\in \{ \pm \}.
\end{align}
\end{theorem}
\def\Sti{\St^\uparrow}

\begin{remark}
If $\tau_2$ is the rotation by $\pi$ about the center of $\fB$, so that $\tau_2(a)=b, \tau_2(b)=a$, then \eqref{re2} is the image of \eqref{re1} under $\tau_2$.
\end{remark}
\def\vt{\vartheta}
\def\bvt{\boldsymbol\vartheta}

\begin{proof} The proof is simple, but we want to give all details here, since we will use a similar proof for the case of an ideal triangle later. The first, and easy, step is to show that $\cSs(\fB)$ is generated by $X$ and the relations \eqref{re1}-\eqref{re2} are satisfied. Then, since we know an explicit $\cR$-basis of $\cSs(\fS)$, an upper estimate argument will finish the proof.

{\bf Step 1.}
For each $k \in \BN$, the set $\partial(k) := \partial (\al(k))$ consists of $k$ points in $a$ and $k$ points in $b$. Let $\St(k)$ be the set of all states $s: \partial(k) \to \{\pm \}$, and $\Sti(k)\subset \St(k)$ be the subset of all increasing states. For $s\in \St(k)$ let $\al(k,s)$ be the stated $\pfB$-tangle diagram, which is $\al(k)$ equipped with state $s$. Similarly, $(\al^k,s)$ is $\al^k$ equipped with $s$. Recall that we have an increasing filtration $\{ F_m(\cSs(\fB))\}$ of $\cSs(\fB)$ and its associated graded algebra $\Gr(\cSs(\fB))$.
By Proposition  \ref{r.basis2}, the set $\{\al(m,s) \mid  s \in \Sti(m)\}$ is an $\cR$-basis of $\Gr_m(\cSs(\fB))$.

By Lemma \ref{r.orderb}, $(\al^m,s) \bequal \al(m,s) \pmod {F_{m-1}(\cSs(\fB))}$, which shows that
$$B_m:=\{(\al^m,s) \mid   s \in \Sti(m)\}$$ is also an $\cR$-basis of $\Gr_{m}(\cSs(\fB))$. It follows that $\{(\al^k,s) \mid k \in \BN , s \in \Sti(k)\}$ is an $\cR$-basis of $\cSs(\fB)$. Since $(\al^k,s)$ is a monomial in the letters in $X$, we conclude that $X$ generates $\cSs(\fB)$ as an $\cR$-algebra.

\FIGc{bigon2a}{An application of Relation \eqref{eq.order}}{2cm}
Let us now prove \eqref{re1}. Apply \eqref{eq.order} as in Figure \ref{fig:bigon2a}, then use \eqref{eq.kink} to remove the kink; we~get
$$ \al(\ve,-) \al(\ve',+) = q^2 \al(\ve,+) \al(\ve', -) +  q^{-1/2} (-q^3) C^\ve_{\ve'},$$
which is \eqref{re1}. The proof of \eqref{re2} is similar.

{\bf Step 2.} Let $A$ be the $\cR$-algebra generated by $X$ subject to the relations \eqref{re1} and \eqref{re2}. Then $A$ is a filtered $\cR$-algebra, where the $m$-th filtration $F_m(A)$ is spanned by the set of monomials in $X$ of degree $\le m$.
The $\cR$-algebra map $\omega: A \to \cSs(\fB)$ defined by $\omega(x)=x$ for all $x \in X$, is a surjective homomorphism of filtered algebras, and induces an algebra homomorphism
$ \Gr(\omega): \Gr(A) \to \Gr(\cSs(\fB)).$
The set $M_m:=\{ \vt_1 \dots \vt_m \mid \vt_j \in X\}$ spans $\Gr_m(A)$.
Presenting each $\vt_j$ as a stated arc on $\fB$, we see that there is state $s\in \St(m)$ such that $\vt_1\dots \vt_m= (\al^m,s)$, and we use this to identify $M_m$ with the set $\{ (\al^m,s) \mid s\in \St(m)\}$.

{\bf Step 3.} Since the second term on the right hand side of \eqref{re1} has degree less than other terms, in  $\Gr(A)$ we have relation \eqref{re1}, with the second term of the right hand side removed:
$$ \al(\ve,-) \al(\ve',+) = q^2 \al(\ve,+) \al(\ve', -). $$
There are states $r,r'\in \St(2)$ such that the left hand side  and the right hand side of the above  are respectively  $(\al^2,r)$ and $(\al^2,r')$, and the above relation can be rewritten as
\be
\lbl{eq.cons1}
(\al^2, r) = q^2 (\al^2, r') \quad \text{in } \ \Gr(A).
\ee
The upshot is that $r'$ is a simple positive permutation of $r$, see Figure \ref{fig:bigon2a}.

{\bf Step 4.}
Let us show that the subset $M^\uparrow(m):=\{(\al^m, s) \mid s \in \Sti(m)\}$ spans $\Gr_m(A)$.
Suppose $\vt_1\dots\vt_m = (\al^m, s) \in M(m)$ with   $\inv(s) >0$.
Then there is a consecutive $s$-decreasing pair $(u,v)\in \partial(m)^2$. Both $u,v$ belong to the same boundary edge, say $b$.
Assume $u$ is an end point of $\vt_j$, then $v$ must be an end point of $\vt_{j+1}$. Then $\vt_j\vt_{j+1}$ look like in the left hand side of Figure \ref{fig:bigon2a}, i.e. $\vt_j\vt_{j+1}$ is exactly the left hand side of \eqref{re1}, or the left hand side of \eqref{eq.cons1}. Replacing  $\vt_j\vt_{j+1}$ by the right hand side of \eqref{eq.cons1},  we get
\be
(\al^m,s) \bequal (\al^m,s') \pmod{F_{m-1}(A)},
\ee
where $s'$ is a simple positive permutation of $s$. An induction on $\inv(s)$ shows that for any $s\in \St(m)$, we have
\be
(\al^m,s) \bequal (\al,\sincr) \pmod{F_{m-1}(A)},
\ee
which, in turns, shows that $M^\uparrow_m$ also spans $\Gr_m(A)$.

Since $\Gr(\omega)(M^\uparrow_m)= B_m$, which is  an $\cR$-basis of $\Gr_m(\cSs(\fB))$, $M^\uparrow_m$ is $\cR$-linearly independent. Thus,  $M^\uparrow_m$ is an $\cR$-basis of $\Gr_m(A)$, and  $\Gr(\omega): \Gr_m(A) \to \Gr_m(\cSs(\fB))$ is an isomorphism. It follows that $\omega: A \to \cSs(\fB)$ is an isomorphism.
\end{proof}

The following lemma is used in the proof of Theorem \ref{r.present.fB}, and we use notations there.
\begin{lemma} \lbl{r.orderb}
Suppose $D$ is a stated $\pfB$-tangle  diagram with $\partial D = \partial(k)$ and each component of $D$ is an arc. Let $s\in \St(k)$ be the state of $D$.

(a) If $D$ contains a returning arc, then, as an element in $\cSs(\fB)$,  $D\in F_{k-1}(\cSs(\fB))$.

(b) If  $D$ has no returning arcs, then, as elements in $\cSs(\fB)$,
$$ D \bequal \al(k,s) \pmod { F_{k-1}(\cSs(\fB))}.$$
\end{lemma}
\begin{proof}

\FIGc{arcreturn}{A returning arc with a double point on it (middle), and two of its resolutions (left and right)}{1.4 cm}
(a) If there is no double point on a returning arc, then $D\in F_{k-1}(\cSs(\fB))$ by relation \eqref{eq.arcs1}.
Suppose there is a double point on a returning arc. Each of the two smooth resolutions of this double point contains a returning arc, see Figure \ref{fig:arcreturn}. The skein relation and induction show that $D\in F_{k-1}(\cSs(\fB))$.

\FIGc{arcreso0}{A double point of 2 non-returning arcs (middle) and its two resolutions (left and right)}{2cm}
(b) If $D$ has no double point, then $ D = \al(k,s)$. Suppose $D$ has a double point. Of the two resolutions of the double point, exactly one does not have a returning arc; see Figure ~\ref{fig:arcreso2}.
By the skein relation, part (a), and induction, we have $ D \bequal \al(k,s) \pmod { F_{k-1}}.$
\end{proof}

\def\fT{\mathfrak T}

\def\ld{\mathrm{lt}}
\begin{proposition}
\lbl{r.domain2}
Suppose $\cR$ is a domain. Then $\cSs(\fB)$ is a domain.
\end{proposition}
\begin{proof}  We say an $\cR$-basis $\{ b_i \mid i\in I\}$ of an $\cR$-algebra $A$ is {\em compatibly ordered}, if   $I$ is a monoid equipped with a total order such that if $i\le i'$ and $j\le j'$ then $i+j \le i'+j'$,   and
$b_i b_j \bequal b_{i+j}$.
 We first proved the following lemma.
\begin{lemma} \lbl{r.domain}
 Suppose $\cR$ is a commutative domain and an $\cR$-algebra $A$  has a compatibly ordered basis. Then $A$ is a domain.
\end{lemma}

\begin{proof} Suppose $x\in A$ is non-zero. Then $x= \sum_i x_i b_i\in A$, with $c_i\in \cR$. The leading term $\ld(x)$ is defined to be $c_jb_j$, where $j$ is the largest index such that $c_j\neq 0$.
Suppose $y\neq 0$ and $\ld(y)= c'_l b_l$.
From the assumptions  $\ld(xy)\bequal c_j c'_l b_{j+l} \neq 0$.
 \end{proof}

Return to the proof of the proposition.
Let $I\subset\BN^3$ be the set of all $\bk=(k,k_a,k_b)\in \BN^3$ such that $k_a,k_b \le k$. For $\bk\in Q$ define $b_\bk=(\al^k, s)$, where $s\in \Sti(k)$ is the only increasing state which has  $k_a$ pluses on edge $a$ and $k_b$ pluses on edge $b$. Then $\{ b_\bk \mid \bk \in I\}$ is an $\cR$-basis of $\cSs(\fB)$. Order $I$ lexicographically.
Lemma \ref{r.orderb} shows that
\be
\lbl{eq.lt}
z(\bk) z(\bk') \bequal z(\bk + \bk') \quad \text{in} \ \Gr(\cSs(\fB)).
\ee
  In other words, $\{ b_\bk \mid \bk \in Q\}$ is a compatibly ordered basis of  $ \Gr(\cSs(\fB))$. By the lemma,  $ \Gr(\cSs(\fB))$ is a domain. Hence, $\cSs(\fB)$ is a domain. \end{proof}

\subsection{Ideal triangle}
Let $\fT$ be an ideal triangle, with boundary edges $a,b,c$ and arcs $\al,\beta, \gamma$ in counterclockwise order, as in Figure \ref{fig:triangle0}.
\FIGc{triangle0}{Ideal triangle $\fT$ (left), with arcs $\al, \beta, \gamma$ (middle), and  $\al({+-})$ (right).}{2cm}

Let $\tau$ be the counterclockwise rotation by $2\pi/3$, so that $\tau(\fT)=\fT$ and $\tau$ gives the cyclic permutation $a \to b \to c \to a$ and $\al \to \beta \to \gamma \to \al$.  For $\ve,\ve'\in \{\pm\}$, let $\al({\ve,\ve'})$ be  $\al$ with  and the state $s$  given by $s(\al\cap c)=\ve, s(\al\cap b)= \ve'$. Let $\beta({\ve,\ve'}) = \tau (\al({\ve,\ve'}))$ and $\gamma({\ve, \ve'}) = \tau^2 (\al({\ve,\ve'}))$. Note that $\tau$ defines an automorphism of the algebra $\cSs(\fT)$.
\def\C{\mathrm{C}}
\def\pfT{\partial \fT}

\begin{theorem}\lbl{r.present.fT} Suppose $\fT$ is an ideal triangle, with the above notations. Then
   $\cSs(\fT)$ is  the $\cR$-algebra generated by the set of twelve generators $$X=\{ \al({\ve,\ve')},\beta({\ve,\ve'}), \gamma({\ve,\ve'})  \mid \ve, \ve'\in \{\pm \}\}$$ subject to the the following relations and their images under $\tau$ and $\tau^2$:
\begin{align}
\lbl{rel1} \beta(\mu,\ve)\, \al(\mu',\ve')& = q \al(\ve, \ve') \,  \beta(\mu,\mu') -q^2 C^{\ve}_{\mu'} \, \gamma(\ve', \mu)
\\
\lbl{rel2} \al(-,\ve)\, \al(+,\ve')& = q^2 \al(+,\ve) \,\al(-,\ve') - q^{5/2} C^\ve_{\ve'}
\\
\lbl{rel2b} \al(\ve,-)\, \al(\ve',+)& = q^2 \al(\ve,+) \,\al(\ve',-) - q^{5/2} C^\ve_{\ve'}
\\
\lbl{rel3} \al(-,\ve)\, \beta(\ve',+)& = q^2 \al(+,\ve)\, \beta(\ve',-) - q^{5/2} \gamma(\ve,\ve')
\\
\lbl{rel4} \al(\ve,-)\, \gamma(+, \ve')& = q^2 \al(\ve,+)\, \gamma(-,\ve') + q^{-1/2} \beta(\ve',\ve).
\end{align}
\end{theorem}

\begin{proof} First we show that $X$ generates $\cSs(\fT)$ and all the relations \eqref{rel1}--\eqref{rel4} are satisfied. Then an upper bound estimate argument will finish the proof.
\FIGc{product}{Diagrams of $\al^2$, $\beta \gamma$, and $\gamma\beta$}{2cm}
Throughout the proof, the order of each $\pfT$-diagram is positive. For example, the diagrams of $\al^2$, $\beta \gamma$, and $\gamma\beta$ are depicted in Figure \ref{fig:product}.

 {\bf Step 1.} Let us show that $X$ generates $\cSs(\fT)$.
 For $\bk=(k_1,k_2,k_3)\in \BN^3$ let $|\bk|:=k_1 + k_2 + k_3$. Let $\theta(\bk)$ be the simple $\pfT$-tangle  diagram which consists of $k_1$ parallels of $\al$, $k_2$ parallels of $\beta$, and $k_3$ parallels of $\gamma$, and $\theta^\bk=\al^{k_1} \beta^{k_2} \gamma^{k_3}$, see Figure \ref{fig:basis1}.
 \FIGc{basis1}{Diagram $\theta(2,3,1)$ (left) and diagram $\al^2 \beta^3 \gamma$ (right)}{2.5cm}

 The set $\partial(\bk):=\partial(\theta(\bk))=\partial(\theta^\bk)$, considered up to isotopy of $\fT$, consists of $k_2+k_3$ points on $a$, $k_1+k_3$ points on $b$, and $k_1+k_2$ points on $c$. Let $\St(\bk)$ be the set of all states  $s: \partial(\bk) \to \{\pm\}$, and
$\Stink\subset \St(\bk)$ be the subset of all increasing states. For $s\in \St(\bk)$ let $\theta(\bk,s)$ be $\theta(\bk)$ with state $s$. Similarly, $(\theta^\bk, s)$ is $\theta^\bk$ with state $s$.
By Proposition~\ref{r.basis2},
$$B_m:=\{ \theta(\bk,s) \mid \bk\in\BN^3, |\bk|=m,  s \in \Stink \}$$
is an $\cR$-basis of $\Gr_m(\cSs(\fT))$.
By Lemma \ref{r.tri3},
$$ \theta(\bk,s) \bequal (\theta^\bk,s) \pmod {F_{|\bk|-1}(\cSs(\fT))},$$
 which implies that $\{ (\theta^\bk,s) \mid \bk \in \BN^3, s \in \Stink \}$ is also an $\cR$-basis of $\cSs(\fT)$. Since each $(\theta^\bk,s)$ is a monomial in $X$,  $X$ generates $\cSs(\fT)$.

 \FIGc{proofrel1}{Proof of Identity \eqref{eq.rel1a}}{2.5cm}
Let now prove all the relations \eqref{rel1}--\eqref{rel4} are satisfied.
Consider \eqref{rel1}. Using the skein relation as in Figure \ref{fig:proofrel1}, we have
\be \lbl{eq.rel1a}
\al(\ve, \ve') \beta(\mu, \mu')= q C^\ve_{\mu'} \gamma(\ve', \mu) + q^{-1} \beta(\mu,\ve) \al (\mu', \ve'),
\ee
which proves \eqref{rel1}.

\FIGc{proofrel2}{Proof of Identity \eqref{eq.rel2a}}{2.5cm}
Now prove \eqref{rel2}. Using Relation \eqref{eq.order} as in Figure \ref{fig:proofrel2} and then \eqref{eq.kink}, we have
\be \lbl{eq.rel2a}
 \al(-,\ve)\, \al(+,\ve') = q^2 \al(+,\ve) \,\al(-,\ve') +  q^{1/2} (-q^3)  C^\ve_{\ve'},
\ee
which proves~\eqref{rel2}. The proof of \eqref{rel2b}--\eqref{rel4} are similar.

\def\barvt{\bar{\vt}}
{\bf Step 2.} Let $A$ be the $\cR$-algebra generated by $X$ subject to the relations \eqref{rel1}--\eqref{rel4}. Then $A$ is a filtered $\cR$-algebra where the $m$-th filtration $F_m(A)$ is spanned by the set of monomials in $X$ of degree $\le m$.
The $\cR$-algebra map $\omega: A \to \cSs(\fB)$, defined by $\omega(x)=x$ for all $x \in X$, is a surjective homomorphism of filtered algebras, and induces an algebra homomorphism
$$ \Gr(\omega): \Gr(A) \to \Gr(\cSs(\fB)).$$
The set $M_m:=\{ \vt_1 \dots \vt_m \mid \vt_j \in X\}$ spans $\Gr_m(A)$. If $\al_i\in \{\al,\beta,\gamma\}$ is  $\vt_i$ without state, then  by presenting each $\vt_j$ as a stated arc on $\fB$, there is a state $s:\partial(\al_1\dots\al_m) \to \{\pm\}$ such that $\vt_1\dots \vt_m= (\al_1 \dots \al_m,s)$. Thus, we can identity
$$M_m=\{ (\al_1\dots \al_m,s) \mid \al_i\in \{\al,\beta,\gamma\},  s:\partial(\al_1\dots\al_m) \to \{\pm\}\}.$$

\def\vM{\vec M}

{\bf Step 3.} Let us now show that the subset $\vM_m\subset M_m$, defined by
$$\vM_m:=\{(\theta^\bk , s) \mid |\bk|=m, s \in \St(\bk)\},$$
 spans $\Gr_m(A)$.
Ignoring the second term of the right hand side of Relation \eqref{rel1} which is of less degree, we get that for any state $r$ of $\al\beta$,
\be
\lbl{eq.ex}
(\beta \al,r)=q (\al \beta,r) \quad \text{in $\Gr(A)$}.
\ee
Together with its images under $\tau$ and $\tau^2$, \eqref{eq.ex} shows that for any permutation $\sigma$ of $\{1,\dots,m\}$,
\be
\lbl{eq.ex3}
(\al_1\dots \al_m,s) \bequal (\al_{\sigma(1)} \dots \al_{\sigma(m)},s) \quad \text{in $\Gr(A)$}.
\ee
In particular, if the numbers of $\al, \beta, \gamma$ among $\al_1,\dots,\al_m$ are components of  $\bk=(k_1, k_2, k_3)$, then $(\al_1\dots \al_m,s) \bequal (\theta^\bk,s)$. This shows the subset $\vM_m$ also spans $\Gr_m(A)$.

\def\vMu{\vec M^\uparrow}
{\bf Step 4.}
Let us show that the subset $\vMu_m\subset \vM_m$, defined by
$$\vMu_m:=\{(\theta^\bk , s) \mid |\bk|=m, s \in \Stink\},$$
 spans $\Gr_m(A)$. First we make the following observation.  Suppose $\vt_1\vt_2$ is the left hand side of one of \eqref{rel2}--\eqref{rel4}, and $\barvt_i\in \{\al,\beta,\gamma\}$ is $\vt_i$ without states. There is a state $r\in \St(\barvt_1 \barvt_2)$ such that
$\vt_1 \vt_2=(\barvt_1 \barvt_2,r)$, and  \eqref{rel2}--\eqref{rel4}, ignoring the second term of the right hand side, give
\be
\lbl{eq.ex1}
(\barvt_1 \barvt_2,r) = q^2 (\barvt_1 \barvt_2,r') \quad \text{in $\Gr(A)$},
\ee
where $r'$ is a simple positive permutation of $r$.


Now assume $\vt_1\dots \vt_m = (\theta^\bk,s)\in \vM_m$  with $\inv(s) >0$. Then there is a consecutive $s$-decreasing pair $(u,v)\in \partial(\bk)^2 $. Let $u$ be an end point of $\vt_i$ and $v$ be an end point of $\vt_j$.
There are two cases:  $\barvt_i=\barvt_j$ and   $\barvt_i \neq\barvt_j$.

 {\em Case 1:} $\barvt_i=\barvt_j$. Say,  $\barvt_i=\barvt_j=\al$. Then one has $j=i+1$ because $u$ and $v$ are consecutive. Besides, $u,v\in b$ or $u,v \in c$.
 \FIGc{order1}{Case 1: we have $\al(-,\ve) \al(+, \ve')$ on the left and $\al(\ve,-) \al(\ve',+)$ on the right}{2.2cm}

 If $u,v \in b$, then  $\vartheta_1 \vartheta_2 =
  \al(-,\ve) \al(+, \ve')$, the left hand side of \eqref{rel2} (see Figure \ref{fig:order1});
   and if  $x_1, x_2 \in c$, then $\vartheta_1 \vartheta_2 =
  \al(\ve,-) \al(\ve',+)$, the left hand side of \eqref{rel2b}.  Using \eqref{eq.ex1}, we get
  \be
  \lbl{eq.res}
   (\theta^\bk,s) \bequal (\theta^\bk,s') \quad \text{in $\Gr(A)$ for some $s'$ with $\inv(s') < \inv(s)$}.
   \ee

{\em Case 2:}  $\barvt_i\neq\barvt_j$. There are 3 subcases: (i) $u,v\in c$, (ii) $u,v \in b$, and (iii): $u,v \in a$. See Figure \ref{fig:order2}.
\FIGc{order2}{Case 2: From left to right we have $\al(-, \ve) \beta(\ve', +)$,
 $\al(\ve,-) \gamma(+, \ve')$, and
 $\beta(-\,\ve) \gamma(\ve',+)$}{2.7cm}

{\em Subcase 2(i):} $u,v\in c$. In this case one has $j=i+1$, and $\vt_i \vt_{i+1}= \al(-, \ve) \beta(\ve', +)$, which is the left hand side of \eqref{rel2} (see Figure \ref{fig:order2}). Again, using \eqref{eq.ex1}, we get \eqref{eq.res}.

Subcase 2(iii)  is similar. Actually  applying the rotation $\tau$, one gets subcase 2(iii) from~2(i).

{\em Subcase 2(ii):} $u,v\in b$. Then $\vt_i \vt_j = \al(\ve,-) \gamma(+, \ve')$, the left hand side of \eqref{rel2b}, see Figure~\ref{fig:order2}. But we might not have $j=i+1$. We only have $i=k_1$ and $j=k_1+k_2+1$, so that $\barvt_i$ is the last $\al$ and $\barvt_j$ is the first $\gamma$ in the product $\theta^\bk= \al^{k_1} \beta^{k_2}\gamma^{k_3}$. However, we can bring $\barvt_i$ next to $\barvt_j$ using \eqref{eq.ex3}: In $\Gr(A)$ we have
$$ (\theta^\bk,s)= (\al^{k_1} \beta^{k_2}\gamma^{k_3},s)\bequal  (\al^{k_1-1} \beta^{k_2} (\al \gamma)\gamma^{k_3-1},s)\bequal (\al^{k_1-1} \beta^{k_2} (\al \gamma)\gamma^{k_3-1},s') \bequal (\theta^\bk,s'),$$
for some state $s'$ with $\inv(s') < \inv(s)$.
Here the second and the last identities are \eqref{eq.ex}, and the third identity follows from \eqref{eq.ex1}.

Thus, in all cases we always have \eqref{eq.res}. An induction shows that for all $(\theta^\bk,s) \in \vM_m$,
\be
(\theta^\bk,s) \bequal (\theta^\bk,\sincr)\quad  \text{in $\Gr(A)$}.
\ee
This shows $\vMu_m$ also spans $\Gr_m(A)$.


{\bf Step 5.} Since $\Gr(\omega)(\vMu_m)= B_m$, which is  an $\cR$-basis of $\Gr_m(\cSs(\fT))$, $\vMu_m$ is $\cR$-linearly independent. Thus,  $\vMu_m$ is an $\cR$-basis of $\Gr_m(A)$, and  $\Gr(\omega): \Gr_m(A) \to \Gr_m(\cSs(\fT))$ is an isomorphism. It follows that $\omega: A \to \cSs(\fB)$ is an isomorphism.
\end{proof}
The following lemma is used in the proof of Theorem \ref{r.present.fT}, and we use notations there.
\begin{lemma} \lbl{r.tri3}
Suppose $D$ is a stated $\pfT$-tangle  diagram with $\partial D = \partial(\bk)$ and each component of $D$ is an arc. Let $s\in \St(\bk)$ be the state of $D$.

(a) If $D$ contains a returning arc, i.e. an arc whose two ends are in one edge of\  $\ \fT$, then, as an element in $\cSs(\fT)$,  $D\in F_{|\bk|-1}$.

(b) If  $D$ has no returning ars, then, as elements in $\cSs(\fT)$,
$$ D = \theta(\bk,s) \pmod { F_{|\bk|-1}}.$$
\end{lemma}
\begin{proof} (a) The proof of Lemma \ref{r.orderb}(a) works also for this case.

(b) \FIGc{arcreso2}{There are two types of double points (up to isotopies and rotations); in each case the only resolution without returning arc is drawn. Other resolutions have returning arcs}{2cm}
If $D$ has no double point, then $ D = \theta(\bk,s)$. Suppose $D$ has a double point. Of the two resolutions of the double point, exactly one does not have a returning arc; see Figure ~\ref{fig:arcreso2}.
By the skein relation, part (a), and induction, we have $ D = \theta(\bk,s) \pmod { F_{|\bk|-1}}.$
\end{proof}

\begin{proposition} \lbl{r.domain3}
Suppose $\cR$ is a domain. Then $\cSs(\fT)$ is a domain. More over, if $\tF$ is a collection of ideal triangles, then $\bigotimes _{\fT\in \tF} \cSs(\fT)$ is a domain.
\end{proposition}
\begin{proof}For $\bk=(k_1,k_2,k_3,k_a,k_b,k_c) \in I$, where
$$ I:=\{ (k_1,k_2,k_3,k_a,k_b,k_c) \in \BN^6 \mid k_a \le k_2 + k_3, k_b \le k_1+ k_3, k_c \le k_1 + k_2 \},$$
 define $b_\bk=(\al^{k_1}\beta^{k_2}\gamma^{k_3}, s)$, where $s\in \Sti(k)$ is the only increasing state such that there are $k_a$ pluses on edge $a$, $k_b$ pluses on edge $b$, and $k_c$ pluses on edge $c$. Then $\{ b_\bk \mid \bk \in I\}$ is an $\cR$-basis of $\Gr(\cSs(\fT))$. We order $I$ using the lexicographic order.
 Lemma \ref{r.tri3} shows that
\be
\lbl{eq.lt3}
z(\bk) z(\bk') \bequal z(\bk + \bk') \quad \text{in} \ \Gr(\cSs(\fT)).
\ee
Lemma \ref{r.domain} shows that    $\Gr(\cSs(\fT)) $ is a domain. Hence $\cSs(\fT)$ is a domain.

By combining the above basis of $\cSs(\fT)$, with the lexicographic order, we get a compatibly ordered basis of  $\bigotimes _{\fT\in \tF} \cSs(\fT)$. Hence $\bigotimes _{\fT\in \tF} \cSs(\fT)$ is a domain.
\end{proof}

\subsection{Zero-divisor. Proof of Theorem \ref{thm.zero}} \lbl{sec.zero}
\begin{proof} If $\pfS$ is a closed manifold, then $\cSs(\fS)=\ooS(\fS)$, and the result is well-known and was proved in \cite{PS2}. Assume  $\pfS\neq \emptyset$. There are only a few simple cases when $\fS$ is not triangulable, listed in (i)--(iv) below. In each case, $\fS=\bfS \setminus \cP$.

(i) $\bfS=S^2$, $|\cP|=1$. Then $\cSs(\fS)= \ooS(\fS)= \cR$, which is domain.

(ii) $\bfS=S^2$, $|\cP|=2$. Then $\cSs(\fS)= \ooS(\fS)= \cR[x]$, which is a domain.
Here $x$ is is only non-trivial simple loop in $\fS$.

(iii) $\bfS $ is a disk and $|\cP|=1 $. Then  $\cSs(\fS)=  \cR$, a domain.

(iv)  $\fS$ is an ideal bigon. Then $\cSs(\fS)$ is a domain by Proposition \ref{r.domain2}.

  Now suppose $\fS$ has an ideal triangulation.  By the triangular decomposition, $\cSs(\fS)$ embeds into
$\bigotimes _{\fT\in \tF} \cSs(\fT)$, which is a domain by  Proposition \ref{r.domain3}. It follows that $\cSs(\fS)$ is a domain.
\end{proof}
\no{
We obtained the following well-known result, which had been proved in \cite{PS2,BW1,CM}
\begin{corollary}
Suppose the ground ring $\cR$ is a domain, and $\fS$ is a punctured bordered surface.
Then $\ooS(\fS)$ is a domain.
\end{corollary}
}
\begin{remark} Besides the case when $\pfS$ is a closed manifold, when $\pfS=\emptyset$, we also have $\cSs(\fS)= \ooS(\fS)$, and the fact that $\ooS(\fS)$ is a domain was known before, see \cite{PS2,BW1,CM}. Our proof in these special cases is different from those in \cite{PS2,BW1,CM}.
Later we show that the Muller skein algebra embeds into $\cSs(\fS)$, hence it is also a domain, a fact proven by Muller before using another method \cite{Muller}.
\end{remark}

\begin{remark}
 Suppose $\cR=\BC$, $q=-1$, and $\fS$ is a triangulated punctured bordered surfaced with a triangulation $\D$.
Then $\cSs(\fS)$ is canonically isomorphic to the ring of regular functions
on the $SL_2(\BC)$-character variety of $\pi_1(\fS)$. The triangular decomosition shows there is a natural embedding
of $\cSs(\fS)$ into $\bigotimes_{\fT\in \tF} \cSs(\fT)$, which is non-commutative. Even in this case of $q=-1$,   the ring $\cSs(\fT)$ and
the triangular decomposition seem new.
\end{remark}

\no{
\subsection{The case $q^2=1$}
 While it is almost obvious from the definition that the ordinary skein algebra $\ooS(\fS)$ is commutative when $q=\pm 1$, it is not quite obvious, because of the boundary relations, that $\cSs(\fS)$ is commutative at $q=\pm 1$.

In what follows we use the notation for $q$-commutator
$$ [x,y]_u := u xy - u^{-1} yx.$$
\begin{proposition}\lbl{r.present.fT1a}
  In $\cSs(\fT)$ we have the following relations
\begin{align*}
\lbl{rel1}[\al  _{++}, \al _{+-}]_q &= [\al _{++}, \al _{-+}]_q=  [\al _{+-}, \al _{--}]_q= [\al _{-+}, \al _{--}]_q= 0 \\
[\al _{+-}, \al _{-+}]_1&= 0\\
[\al _{++},  \al _{--}]_{q^2} &= q^2- q^{-2}\\
[\al _{+ \, \ve}, \beta _{\ve' +}]_{q^{1/2}}&=[\al _{-\, \ve}, \beta _{\ve' -}]_{q^{1/2}} =[\al _{-\, \ve}, \beta _{\ve'  +}]_{q^{1/2}} =0 \quad \text{for all} \ \ve, \ve' \in \{\pm\}\\
[ \al _{+\, \ve} \, \beta _{\ve'-}]_{q^3/2} &= (q^{2} - q^{-2}) \gamma_{\ve\ve'} \quad \text{for all} \ \ve, \ve' \in \{\pm\}
\end{align*}
and all their images under the cyclic permutation $\tau$.
\end{proposition}

\begin{corollary}
When $q=\pm1$, the algebra
\end{corollary}
}

\def\tY{\tilde{\cY}}
\def\id{\mathrm{id}}

\section{Chekhov-Fock algebra and quantum trace} \lbl{sec.QT}


\def\btau{{\tilde \tau}}
\def\tfT{{\tilde \fT}}
\def\tF{{\tilde \cF}}
\subsection{Chekhov-Fock triangle algebra, Weyl normalization}

Suppose $\fT $ is an ideal triangle with boundary edges $a,b,c$ and arcs $\al, \beta, \gamma$  as in Figure~\ref{fig:triangle0}.
Define $\cY(\fT  )$ to be the $\cR$-algebra
$$ \cY(\fT  ) = \cR\la y_a^{\pm 1},  y_b^{\pm 1},  y_c^{\pm 1} \ra / (y_a y_b = q \,y_b y_a, y_b y_c = q\, y_c y_b, y_c y_a = q \, y_c y_a).$$
Then $\cY(\fT)$ belongs to a type of algebras called {\em quantum tori}, see eg. \cite[Section 2]{Le:QT}.

Suppose $x,y$ are elements of an $\cR$-algebra such that $x y= q^k yx$, where $k\in \BZ$. Define the Weyl normalization of $xy$ by
$$ [xy]:= q^{-k/2} xy = q^{k/2} yx.$$
The advantage is that $[xy]=[yx]$. For example, for $b,c\in \cY(\fT)$ and $\ve, \ve' \in \{ \pm 1\}$, we have
\be \lbl{eq.weyl}
[(y_c)^{\ve} (y_b)^{\ve'}]= q^{\ve  \ve'/2 } (y_c)^{\ve } (y_b)^{\ve' }.
\ee

\lbl{sec.tri1}
  The rotation $\tau: \fT \to \fT$, which gives the cyclic permutations $\al \to \beta\to \gamma\to \al$ and $ a\to b \to c \to a$, induces algebra automorphisms of the
 algebras $\cSs(\fT)$ and $\cY(\fT)$.
\begin{proposition}
There exists a unique $\cR$-algebra homomorphism $\phi: \cSs(\fT) \to\cY(\fT)$ which is $\tau$-equivariant and
satisfies
\begin{align}
\lbl{eq.qt}
\phi(\al(\ve \ve')) = \begin{cases}
0 \quad & \text{if } \ \ve=-, \ve'=+ \\
[c^{\ve} b^{\ve'}] &\text{otherwise}.
\end{cases}
\end{align}
\end{proposition}
\begin{proof} The proof follows from an easy checking that  the definition \eqref{eq.qt}
respects all the defining relations of $\cSs(\fT)$ described in Proposition \ref{r.present.fT}.
\end{proof}

\def\trD{\mathrm{Tr}_\Delta}

\def\vkD{\varkappa_\D}
\def\vkDp{\varkappa_{\D'}}
\subsection{Quantum trace}
Let $\D $ be a triangulation of a punctured bordered surface $\fS$ and  $\tF=\tF(\D)$
 be the collection of disjoint ideal triangles obtained by splitting $\fS$ along the edges
of $\D$, see Section \ref{sec.tri}. Let $\cE$ be the set of all edges of $\D$,
and $\tE$ be the set of all edges of all triangles in $\tF$.

 Using the triangular decomposition \eqref{eq.tri} and the algebra map $\phi$ of Section \ref{sec.tri1},
 define $\vkD$ as the composition
 $$ \vkD : \cSs \overset {\rho_\D} \longrightarrow \bigotimes _{\fT \in \tF(\D)} \cSs(\fT)
 \overset {\otimes \phi} \longrightarrow  \cY(\tF):=\bigotimes _{\fT \in \tF(\D)} \cY(\fT) .$$

Bonahon and Wong \cite{BW1} constructed an algebra homomorphism (quantum trace map)
$$ \trD: \hcSs(\fS) \to \YD,$$
where $\YD$, a version of the Chekhov-Fock algebra, is an $\cR$-subalgebra of $\cY(\tF)$. For now, we consider $\trD$ as a map with target $\cY(\tF)$. We recall the definition of $\YD$ in Section~\ref{sec.YD}.



\def\hSS{\widehat \cSs(\fS)}
\def\htD{\widehat\vkD}
 \begin{theorem}
 If $\D$ is an ideal triangulation of a punctured bordered surface $\fS$, then
the composition $\widehat\vkD : \widehat \cSs(\fS) \to \cSs(\fS) \overset {\vkD} \longrightarrow \cY(\tF) $ is equal to the quantum trace map of Bonahon and Wong.
 \end{theorem}
 \begin{proof} (i) The case $\fS=\fT$, an ideal triangle with notations of Figure \ref{fig:triangle0}. In this case $\hcSs(\fT)$ is generated by $\al(\ve,\ve')$, the arcs $\delta(\ve,\ve')$, and their images under $\tau, \tau^2$. Here $\delta(\ve,\ve')$ is a returning arc with states $\ve,\ve'$.
 For each of these generators, the image of $\trD$ described in \cite[Theorem 11]{BW1} is exactly the image of $\vkD$ given by \eqref{eq.arcs1} and \eqref{eq.qt}. Hence $\htD=\trD$.

 (ii) Now return to the case of general punctured surfaces.
  Suppose $e_1 , e_2 $ are edges of $\fS$ and $\fS'=\fS/(e_1 =e_2 )$. Then $\fS'$ inherits a triangulation $\D'$ from $\D$,
where the set $\cE'$  of edges of $\D'$ is the same as $\cE$, except that the two edges $e_1 ,e_2 $
of $\cE$ are glued
together, giving an edge $e$ of $\cE'$. Both $\fS$ and $\fS'$ are obtained from the same collection $\tF$ of ideal triangles by  identifications of edges, with $\fS'$ having one more identification.

By \cite[Theorem 11]{BW1}, $\trD$ is uniquely characterized by its values for ideal triangles and the
the following condition: For any such pair $\fS, \fS'$ and any stated $\pfS'$-tangle $\al$ vertically transversal to $e$,
\be
\lbl{eq.con1}
 \tr_{\D'}(\al) =\sum_{\beta}
\trD(\beta),\ee
where $\beta$ runs the set of all lifts of $\al$, see Section \ref{sec.31}. The sequence of maps
$$ \cSs(\fS') \overset \rho \longrightarrow \cSs(\fS) \overset {\rho_\D} \longrightarrow \bigotimes _{\fT \in \tF(\D)} \cSs(\fT)
 \overset {\otimes \phi} \longrightarrow  \cY(\tF)$$
 shows that $\vkDp = \vkD \circ \rho$. Hence, from the definition of $\rho$, we have
 $$
 \widehat{\varkappa_{\D'}} (\al) =\sum_{\beta}
 \htD(\beta),
$$
i.e. $\htD$ also satisfies the above condition \eqref{eq.con1}. This proves $\htD= \trD$.
\end{proof}
  We have seen that the triangular decomposition \eqref{eq.tri} gives a simple proof of the existence of the quantum trace map of Bonahon and Wong \cite{BW1}. 
 Because there is no analog of $\rho$ relating $\hcSs(\fS')$ and $\hcSs(\fS)$, in \cite{BW1} the quantum trace map has to be defined directly on $\hcSs(\fS)$, and the proof of well-definedness involves difficult calculations.
 For yet another proof of the existence of the quantum trace map, the reader can consult \cite{Le:QT}.

\subsection{Chekhov-Fock algebra}\lbl{sec.YD}
 We continue with the notation of the previous subsection.
For each $\fT\in \tF$ we consider $\cY(\fT)$ as a subalgebra of $\cY(\tF)=\bigotimes _{\fT \in \tF(\D)} \cY(\fT)$
 under the natural embedding. Then $\cY(\fT)$ is the
$\cR$-algebra generated by $y_e^{\pm 1}$ with $e \in \tE$,
subject to the relation $y_{e_1} y_{e_2}= y_{e_2} y_{e_1}$ if $e_1$ and $e_2$ are
edges of different triangles, and $ y_{e_1} y_{e_2}= q y_{e_2} y_{e_1}$ if $e_2$
and $e_1$ are edges of a triangle and $e_2$ follows $e_1$ in counterclockwise order.

Under the natural projection
 $\pr:\tE \to \cE=\cE(\D)$, an edge $e\in \cE$ has  one or two pre-images in $\tE$; each is called
 a {\em lift} of $e$. For each $e\in \cE$
define $y_a\in \cY(\fT)$ by
 \begin{itemize}
\item If $e$  has a unique lift $e'\in \tE$, then $y_e=y_{e'}$.

\item If $e$  has two lifts $e', e'' \in \tE$, then $y_e = [y_{e'} y_{e''}]$.
\end{itemize}

Let $\YD$ be the $\cR$-subalgebra of $\cY(\fT)$ generated
by $y_a^{\pm 1}, a\in \cE$. Then $\YD$ is a version of the (multiplicative) Chekhov-Fock
algebra of $\fS$ associated with the triangulation $\D$. In \cite{BW1}, it proved that the image of $\trD$ is in $\YD$, which can also be proved easily from the definition of $\vkD$.

\def\oD{{\mathring {\D}}}

\section{
Relation with the skein algebra of Muller}
\lbl{sec.Muller}

\def\SP{(\bfS,\cP)}
\subsection{Skein algebra of marked surface} Recall that a  {\em marked surface} $ (\bfS,\cP)$ consists of  a compact oriented 2-dimensional manifold $\bfS$ with (possibly empty) boundary $\pbfS$ and a finite set  $\cP \subset \pbfS$. We recall the definition of the Muller skein algebra \cite{Muller}, following \cite{Le:QT}.

Let $\fS= \bfS \setminus (\cP  \cup \partial'(\bfS))$, where $\partial'(\bfS)$ is the union of all connected components of $\pbfS$ which do not intersect $\cP$. Then $\fS$ is a punctured marked surface.

{\em A $\cP$-tangle} $\al$ is defined just like a $\pfS$-tangle, only with $\partial(\al) \subset \cP\times (0,1)$. More precisely, a {\em $\cP$-tangle} $\al$ is
a compact, framed, properly embedded 1-dimensional non-oriented smooth submanifold $\al$ of $\bfS \times (0,1)$ such that $\partial (\al) \subset \cP\times (0,1)$ and
 at every boundary point of $\al$ the framing is vertical.
 Two  $\cP$-tangles are {\em isotopic} if they are
 isotopic through the class of $\cP$-tangles. Define
 $\SMuller\MN$ to be   the $\cR$-module freely spanned by isotopy classes of
 $\cP$-tangles modulo  the skein relation \eqref{eq.skein},  the trivial loop relation \eqref{eq.loop},
and
 the new {\em trivial arc relation} (see Figure \ref{fig:RelArc}).
 \FIGc{RelArc}{Trivial arc relation}{2cm}

 More precisely, the trivial arc relation says $\al=0$ for any $\cP$-tangle $\al$ of the form $\al= \al' \sqcup a $, where $a\subset \bfS \times (0,1)
    \setminus \al'$ is  an arc with two end points  in $p\times (0,1)$ for some  $p\in \cP$, such that
     $a$ and the part of $p\times (0,1)$ between the two end points of $a$ co-bound a disk in $\bfS \times (0,1)
    \setminus \al'$.

     As usual, the product of two $\cP$-tangles is obtained by stacking the first on top of the second. With this product, $\SMuller\MN$ is an $\cR$-algebra.

\def\tO{\tilde \Omega}
    Let $\cSsp(\fS) \subset \cSs(\fS)$ be the $\cR$-submodule spanned by stated $\pfS$-tangles whose states are all $+$.  For an $\cP$-tangle $\al$ define a stated $\pfS$-tangle $\tO(\al)$ by moving all the boundary points of $\al$  slightly to the left (i.e. along the positive direction of $\partial \bfS$), keeping the same height,  and equipping $\tO(\al)$ with state $+$ at every boundary points. Relations \eqref{eq.skein}--\eqref{eq.arcs} show that $\tO$ descends to a well-define
  $\cR$-linear map
    $$\Omega: \SMuller(\bfS, \cP) \to \cSsp(\fS).$$
  \begin{proposition}
  The map $\Omega$ is an $\cR$-algebra isomorphism.
  \end{proposition}
  \begin{proof} It is easy to see that $\cSsp(\fS)$ is the $\cR$-module freely spanned by $\pfS$-tangles with $+$ states, subject to those relations from \eqref{eq.skein}--\eqref{eq.order} which involve only $+$ states; namely relations \eqref{eq.skein}, \eqref{eq.loop}, and the middle relation of \eqref{eq.arcs}. Since $\tO$ maps the set of isotopy classes of $\cP$-tangles isomorphically onto the set of isotopy classes of $+$ stated $\pfS$-tangles, and maps the defining relations of $\SMuller\MN$ onto the defining relations of $\cSsp(\fS)$, it induces an isomorphism $\Omega: \SMuller(\bfS, \cP) \to \cSsp(\fS).$

  Alternatively, the
  $\cR$-basis of $\SMuller(\bfS,\cP)$,  given explicitly in \cite[Lemma 4.1]{Muller}, is mapped by $\Omega$ to  the $\cR$-basis of $\cSsp(\fS)$ given in Theorem \ref{thm.basis}, with all $+$ states. This shows $\Omega$ is an  isomorphism. \end{proof}
\def\bvkD{\bar \varkappa_\D}
 \def\oE{\mathring {\cE}}
\def\ooE{\mathring E}

\subsection{Proof of Theorem \ref{thm.inj}}
 \begin{proof}   Recall that  we define the  quantum trace map on $\SMuller(\bfS,\cP)$
   $$ \bvkD: \SMuller(\bfS,\cP) \overset \Omega \longrightarrow  \cSsp(\fS)  \overset \vkD \longrightarrow  \YD.$$

Number the set $\cE$ of edges of $\D$ so that $\cE=\{e_1,\dots,e_n\}$.  For $\bk=(k_1,\dots,k_n)\in \BZ^n$ let
$$ y^{\bk}:= (y_{e_1})^{k_1} (y_{e_2})^{k_2} \dots (y_{e_n})^{k_n} \in \YD.$$
 The set $\{ y^{\bk} \mid \bk\in \BZ\}$ is an $\cR$-basis of $\YD$. We order all $y^{\bk}$ using the lexicographic order of $\bk\in \BZ^n$, and use this order to define the {\em leading term} $\ld(x)$ of any $0 \neq x\in\YD$.
\def\ptS{\partial \tilde{\fS}}
\def\tS{\tilde {\fS}}

Let $\tS= \bigsqcup_{\fT\in \tF} \fT$, and $\pr: \tS \to \fS$ be the natural projection.
Suppose $D $ is a simple $\pfS$-tangle diagram which is {\em $\D$-normal}, i.e. it is $e$-normal for all edge $e$ of $\D$ (see Proof of Theorem \ref{thm.1a}(b)). Then $\tD:=\pr^{-1}(D )$ is a $\ptS$-tangle diagram consisting of non-returning arcs in ideal triangles in $\tF$. Define $\bk_D =(k_{D ,1},\dots,k_{D ,n})\in \BZ^n$, where $k_{D ,i} = |D  \cap e_i|$. Then $\bk_D $ totaly determines the isotopy class of $D $. Every
 simple $\pfS$-tangle diagram is isotopic to a $\D$-normal one.

An edge $e$ of $\D$ is {\em interior} if it is not a boundary edge of $\fS$. Let $\ooE$ be the union of all the interior edges of $\D$.
Fix  an orientation of $\ooE$ and provide the boundary edges of $\fS$ with the positive orientation.
Lift these orientations to edges of $\tE$, and get an orientation $\ori$ of $\partial\tS$.
Suppose $D $ is a positively ordered, $\D$-normal, simple $\pfS$-tangle diagram, with $+$ states. For  every map $r:D  \cap \ooE \to \{\pm\}$, let $\tD(r)$ be $\tD$ with $\ori$-order, and states defined by  the lift of $r$ on $\ooE$ and $+$ on all other edges. From the definition, we have the following state sum
\be
\lbl{eq.state}
\vkD(D ) = \sum_{r:D  \cap \ooE \to \{\pm\}} \phi(\tD(r)).
\ee
Let $r_+$ be the map $r_+:D  \cap \ooE \to \{+\}$. Since $\tD(r_+)$ consists of non-returning arcs with $+$ states, \eqref{eq.qt} and Lemma \ref{r.tri3} show that $\phi(\tD(r_+))$ is non-zero and
$$\ld(\phi(\tD(r_+)))\bequal y^{\bk_D }.$$
More over  \eqref{eq.qt} and Lemma \ref{r.tri3} show that the leading term of any $\phi(\tD(r))$, with $r \neq r_+$, is smaller than that of $\phi(\tD(r_+))$. Hence
   \be
    \ld(\vkD(D )) \bequal y^{\bk_D }.
    \ee
   From here it is easy to see that the image under $\vkD$
 of a non-trivial linear combination of $+$ stated, positively ordered, $\D$-normal,  simple $\pfS$-tangle diagrams is non-zero, which proves that $\bvkD$ is injective.
  \end{proof}


\end{document}